%% file: main.tex
\documentclass{article}

\usepackage[colorlinks,linkcolor=red,citecolor=blue, pagebackref=true]{hyperref}
\usepackage[margin=1.1in]{geometry}
\usepackage[numbers]{natbib}
\usepackage{booktabs}

\input{preamble}

\input{macros}
\usepackage[shortlabels]{enumitem}

\usepackage{authblk}

\title{Concentration inequalities for strong laws\\and laws of the iterated logarithm}
\author{
  Johannes Ruf$^\dagger$ and Ian Waudby-Smith$^\ddagger$
  
}

\begin{document}

\maketitle
\begingroup
\renewcommand{\thefootnote}{\fnsymbol{footnote}}
\footnotetext[2]{Department of Mathematics, London School of Economics and Political Science, London WC2A 2AE, UK} %
\footnotetext[3]{Department of Statistics and The Miller Institute, University of California, Berkeley, Berkeley, CA 94720, USA}   %
\endgroup
\setcounter{tocdepth}{1}
\makeatletter
\renewcommand\tableofcontents{%
  \@starttoc{toc}%
}

\makeatother

\begin{abstract}
\input{abstract}

\end{abstract}

\input{content}

\input{proofs}

\subsection*{Acknowledgments}
The authors thank Sivaraman Balakrishnan, Rafael Frongillo, Aaditya Ramdas, and anonymous referees for helpful discussions.

\bibliographystyle{plainnat}
\bibliography{references.bib}
\newpage
\appendix

\input{composite}

\end{document}

%% file: preamble.tex
\usepackage[utf8]{inputenc}
\usepackage{amsmath, amsthm, amssymb, amsfonts}
\usepackage{dsfont}
\usepackage{graphicx}
\usepackage[makeroom]{cancel}
\usepackage{mathabx}
\usepackage{xcolor}
\usepackage{array}
\usepackage{colortbl}
\usepackage{tcolorbox}
\usepackage{accents}
\usepackage{mathrsfs}
\usepackage{tikz}
\usetikzlibrary{positioning, arrows.meta}

\usepackage[normalem]{ulem}

\usepackage{autonum}

\usepackage{thmtools}
\declaretheorem[name=Theorem,numberwithin=section]{theorem}
\declaretheorem[name=Proposition,sibling=theorem]{proposition}
\declaretheorem[name=Lemma,sibling=theorem]{lemma}
\declaretheorem[name=Corollary,sibling=theorem]{corollary}
\declaretheorem[name=Remark,sibling=theorem]{remark}
\declaretheorem[name=Definition,sibling=theorem]{definition}

\usepackage[capitalise,noabbrev]{cleveref}
\hypersetup{hypertexnames=false}

%% file: macros.tex
\newcommand{\RR}{\mathbb R}
\newcommand{\EE}{\mathsf E}

\newcommand{\NN}{\mathbb N}

\newcommand{\U}{\mathsf{U}}
\renewcommand{\P}{\mathsf{P}}

\def\ddefloop#1{
  \ifx
    \ddefloop#1
  \else\ddef{#1}
    \expandafter
    \ddefloop
  \fi}
\def\ddef#1{\expandafter\def\csname #1cal\endcsname{\ensuremath{\mathcal{#1}}}}

\ddefloop
ABCDEFGHIJKLMNOPQRSTUVWXYZ
\ddefloop

\newcommand{\1}{\mathds{1}}

\newcommand{\infseqn}[1]{(#1)_{n \in \NN}}
\newcommand{\infseqnz}[1]{(#1)_{n \in \NN_0}}

\newcommand{\dd}{\mathrm{d}}

\newcommand{\iid}{\text{i.i.d.}}

\newcommand{\eps}{\varepsilon}
\newcommand{\Var}{\mathrm{Var}}

\defcitealias{komlos1975approximation}{KMT-I}
\defcitealias{komlos1976approximation}{KMT-II}

\newcommand{\supP}{{\sup_{\P \in \Pcal} \,}}
\newcommand{\Pin}{{\P \in \Pcal}}

\newcommand{\supkm}{{\sup_{k \geq m}}}
\newcommand{\tth}{{\text{th}}}
\newcommand{\mto}{{m \nearrow \infty}}
\newcommand{\xto}{{x \nearrow \infty}}
\newcommand{\ntoinfty}{{n \nearrow \infty}}

\newcommand{\toinfty}{\nearrow \infty}
\newcommand{\tozero}{\searrow 0}

\newcommand{\sigmaub}{{\widebar \sigma}}
\newcommand{\sigmasqub}{{\widebar \sigma^2}}

\newcommand{\brackone}{{(1)}}
\newcommand{\brackq}{{(q)}}
\newcommand{\bracktwo}{{(2)}}

\newcommand{\onediv}{{\mathrm{1\text{-}div}}}
\newcommand{\qdiv}{{q\mathrm{\text{-}div}}}
\newcommand{\fluc}{{\mathrm{fluc}}}

\newcommand{\brackeps}{{(\eps)}}
\newcommand{\brackmeps}{{(m,\eps)}}

\newcommand{\Pcalbar}{\overline \Pcal}

\newif\ifverbose %
\verbosefalse %

\newcommand{\verbose}[1]{
  \ifverbose\textcolor{gray}{\textit{v $\rightarrow$}}\quad #1
    \textcolor{gray}{\textit{nv $\rightarrow$}}\quad
  \fi}

%% file: abstract.tex
  We derive concentration inequalities for sums of independent and identically distributed random variables that yield non-asymptotic generalizations of several strong laws of large numbers including some of those due to Kolmogorov [1930], Marcinkiewicz and Zygmund [1937], Chung [1951], Baum and Katz [1965], Ruf, Larsson, Koolen, and Ramdas [2023], and Waudby-Smith, Larsson, and Ramdas [2024]. As applications, we derive non-asymptotic iterated logarithm inequalities in the spirit of Darling and Robbins [1967], as well as pathwise (sometimes described as ``game-theoretic'') analogues of strong laws and laws of the iterated logarithm.

%% file: content.tex
\section{Introduction}

\verbose{
Laws of large numbers are foundational results in probability and statistics. They are found in the derivation of statistical procedures, the analysis of randomized algorithms, and within the canonical toolbox of probability theory itself. Indeed, they are ubiquitous in several disciplines ranging from the abstract and mathematical to the concrete and scientific. This paper sheds some new light on laws of large numbers by developing a single family of sharp inequalities that recover many existing results from the literature as immediate consequences.
}

The strong law of large numbers (SLLN) due to \citet{kolmogorov1930loi} states that for independent and identically distributed (\iid{}) random variables $\infseqn{X_n}$ on a probability space $(\Omega, \Fcal, \P)$ with expected value zero, their partial sample averages converge to zero with $\P$-probability one, i.e.  
\begin{equation}\label{eq:intro-kolmogorov-slln}
  \frac{S_n}{n} = o(1)\quad \text{with $\P$-probability one,}
\end{equation}
where $S_n := \sum_{i=1}^n X_i$.
The SLLNs due to \citet{marcinkiewicz1937fonctions} state that if in addition, the $q^\tth$ moment of $X_1$ is bounded for some $q \in [1,2)$, i.e.~$\EE_\P|X_1|^q < \infty$, then the rate of convergence in \eqref{eq:intro-kolmogorov-slln} can be upgraded to $o(n^{1/q - 1})$, meaning
\begin{equation}\label{eq:intro-mz-slln}
  \frac{S_n}{n^{1/q}} = o(1)\quad \text{with $\P$-probability one.}
\end{equation}
The quote \textit{``behind every limit theorem is an inequality''} is often attributed to Kolmogorov; see \citep[p.~216]{tropp2023probability}. In this spirit, \cref{section:slln} states non-asymptotic and time-uniform concentration inequalities that immediately yield \eqref{eq:intro-kolmogorov-slln} and \eqref{eq:intro-mz-slln}. Moreover, since the constants therein are explicit, one can consider a family of probability measures for which certain transformations of $X_1$ are uniformly integrable, leading to non-asymptotic generalizations of some existing uniform SLLNs in the literature \citep{chung_strong_1951,waudby2024distribution}. Here, ``uniformity'' is meant in the sense of \citet{chung_strong_1951} and this term will be recalled precisely in \cref{corollary:chung-slln}. Furthermore, these concentration inequalities are sufficiently sharp to conclude SLLNs in the spirit of \citet{baum1965convergence}, iterated logarithm inequalities in the spirit of \citet{darling1967iterated} under only a finite second moment assumption, and pathwise SLLNs and LILs.
\verbose{
Despite these implications, our proofs proceed through different but elementary arguments, relying on properties of time-uniform supermartingale concentration.
}

\verbose{
In the case of random variables having finite variances, \cref{section:lil} combines the aforementioned concentration inequalities with a result of \citet{howard_exponential_2018} to arrive at an iterated logarithm inequality that immediately implies a distribution-uniform generalization of the upper bound in the law of the iterated logarithm (LIL) due to \citet{kolmogorov1929uber} and \citet{hartman1941law}. The same section also extends the iterated logarithm inequalities of \citet{darling1967iterated} to random variables without moment generating functions and leads to an improvement of an LIL found in \citet[Theorem~6]{baum1965convergence}. The aforementioned applications to Baum-Katz-type SLLNs and LILs can be found in \cref{section:baum-katz}.

Finally, \cref{section:game-theoretic} applies the inequalities of the preceding sections to provide pathwise (sometimes referred to as ``game-theoretic'') versions of SLLNs and LILs. These applications yield a strengthening and an extension of a SLLN found in \citet[Theorem~4.3]{ruf2022composite} to higher moments and to LILs. The same section also introduces a notion of \emph{uniform} pathwise SLLNs and LILs on lattices of events, ultimately showing that it is equivalent to the familiar notion of uniformity considered by \citet{chung_strong_1951} when the lattice is chosen appropriately. This equivalence result closes a gap that has existed between existing uniform and pathwise strong laws in the literature.
}

The literature on laws of large numbers is vast and several extensions and refinements have been seen even in recent years. See \citet{gut1978marcinkiewicz,gut1980convergence}, \citet{gut2011intermediate}, \citet{lanzinger2003baum,lanzinger2004refined}, and \citet{neri2025quantitative} for advances on Baum-Katz-type laws of large numbers. See \citet{rakhlin2015sequential} for SLLNs for empirical processes that are uniform with respect to a class of functions. See also \citet{karatzas2023weak,karatzas2023strong} and \citet{berkes2025hereditary} for laws of large numbers in hereditary convergence and \citet{vovk2025law} for a non-asymptotic weak law. After acceptance of the present paper, we learned of the work of \citet{nguyen2021maximal} who derive a maximal inequality in \citep[Corollary~3]{nguyen2021maximal} for the same quantity that we consider in \cref{theorem:lq-concentration}, but via a different proof technique. Finally, we remark that obtaining non-asymptotic inequalities that yield familiar asymptotic results has been the object of study in the context of central limit theorems and strong invariance principles \citep{katz1963note,austern2022efficient,ye2025computable,waudby2025nonasymptotic}. 

\subsubsection*{Notation and conventions} Throughout, we will let $(\Omega, \Fcal)$ denote a fixed measurable space equipped with a stochastic process $\infseqn{X_n}$. Throughout, $\Pcalbar$ will denote the family of all probability measures for which $\infseqn{X_n}$ is a sequence of \iid{} integrable random variables with expected value zero. For any probability measure $\P$, we denote the expected value of a $\P$-integrable random variable $Y$ by $\EE_\P[Y]$ and its variance, provided it exists, by $\Var_\P[Y]$. We use the convention throughout that $0/0 = 0$. We define $\log \log n$ as $\log \log (n \lor 3)$ whenever $n \in \{1,2\}$.

\section{Concentration for strong laws of large numbers}\label{section:slln}
Let us first derive a non-asymptotic concentration inequality that implies both Kolmogorov's and Chung's SLLNs in \eqref{eq:intro-kolmogorov-slln} and \cref{corollary:chung-slln}, respectively.

\begin{theorem}
[$L^1$ concentration for the strong law of large numbers]
\label{theorem:l1-concentration}
Fix $\P \in \Pcalbar$ and define the truncated first absolute moment $\U_\P(x)$ of $X_1$ under the measure $\P$ at a lower truncation level $x$ as
  \begin{equation}\label{eq:l1-uniform-integrability-tail}
    \U_\P(x) := \EE_\P \left [ |X_1|\1 \{ |X_1| \geq x \} \right ];\qquad x \geq 0.
  \end{equation}
  For any $\lambda \in (0, 1/2)$, $\eps > 0$, and $m \in \NN$, we have
  \begin{equation}\label{eq:l1-concentration}
    \P \left [ \supkm \frac{|S_k|}{k}  \geq \eps \right ] \leq \frac{262}{\eps^2 \wedge 1} \left ( m^{2\lambda - 1} + \U_\P(m^\lambda) \right ). 
  \end{equation}
\end{theorem}
Let us observe how Kolmogorov's SLLN follows immediately from \cref{theorem:l1-concentration}. Note that for any probability measure $\P \in \Pcalbar$ and any $x \geq 0$, it holds that $\U_\P(x) < \infty$. Hence the right-hand side of \eqref{eq:l1-concentration} vanishes as $\mto$ by monotone convergence. Moreover, recall that for any sequence $\infseqn{Y_n}$ on $(\Omega, \Fcal, \P)$, we have $Y_n \to 0$ with $\P$-probability one if and only if for every $\eps > 0$, $\P[ \supkm |Y_k| \geq \eps ] \to 0$ as $m \nearrow \infty$. Putting these two observations together, we see that $  S_n / n  \to 0$ as $n \nearrow \infty$ with $\P$-probability one as a consequence of \eqref{eq:l1-concentration}.
In fact, \cref{theorem:l1-concentration} can be used to derive the following stronger distribution-uniform generalization of Kolmogorov's SLLN due to \citet{chung_strong_1951}.

\begin{corollary}[Chung's distribution-uniform $L^1$ SLLN \citep{chung_strong_1951}]\label{corollary:chung-slln}
Let $\Pcal \subset \Pcalbar$ be a collection of probability measures for which  $X_1$ is $\Pcal$-uniformly integrable, meaning
\begin{equation}
  \lim_{\xto} \supP \EE_\P \left [ |X_1| \1 \{ |X_1| \geq x \} \right ] = 0.
\end{equation}
Then the SLLN holds uniformly in $\Pcal$, meaning that for any $\eps > 0$,
\begin{equation}\label{eq:chung}
  \lim_\mto \supP \P \left [ \supkm  \frac{|S_k|}{k}  \geq \eps \right ] = 0.
\end{equation}
\end{corollary}

In light of the expression in \eqref{eq:l1-concentration} combined with the discussion thereafter, one can observe that Chung's result in \eqref{eq:chung} follows from \cref{theorem:l1-concentration} by setting $\lambda = 1/4$ (say) and taking a supremum over $\P \in \Pcal$ on both sides of \eqref{eq:l1-concentration}. Note that by \citet[Theorem~1]{waudby2024distribution}, $\Pcal$-uniform integrability of $X_1$ is also a necessary condition for \eqref{eq:chung} to hold so \cref{theorem:l1-concentration} is sharp in the sense that it can be used to conclude pointwise and uniform SLLNs without relying on additional assumptions.
\verbose{While $\lambda \in (0, 1/2)$ can be adjusted in the expression of \eqref{eq:l1-concentration} to provide an explicit and potentially sharper rate of convergence for the probability in \eqref{eq:chung} depending on the speed of uniform decay of the truncated first moment $\U_\P$, we do not dwell on such adjustments here since the inequality to be provided shortly in \cref{theorem:lq-concentration} will enjoy exponential (rather than polynomial) dependence on $m$.}

Let us now provide a short proof of \cref{theorem:l1-concentration}, which relies on an $L^1$ line-crossing inequality due to \citet[Theorem~4.1]{ruf2022composite}, after which we will move on to the case of random variables in $L^q$ for $q \in [1,2)$ where alternatives to this line-crossing inequality will be developed in a bespoke fashion.

\begin{proof}[Proof of \cref{theorem:l1-concentration}]
  Let $\lambda \in (0, 1/2)$, $\eps > 0$, and $m \in \NN$ be arbitrary. By monotonicity of $\U_\P$,
  \begin{align}
    \P \left [ \sup_{k \geq m}  \frac{|S_k|}{k}   \geq \eps \right ] &\leq \1\left \{ \U_\P(m^\lambda) > \frac{\eps}{4}\right  \} + \P \left [ \supkm  \frac{|S_k|}{k}  \geq  \frac{\eps}{2} + 2\U_\P (m^\lambda) \right ].\label{eq:l1-proof-first-prob}
  \end{align}
  We rely on a line-crossing inequality of
  \citet[Theorem~4.1]{ruf2022composite}, which states that for any $\gamma > 0, x \geq 1$,
  \begin{equation}
    \P \left [ \sup_{k \in \NN} \frac{\lvert S_k  \rvert}{k + \gamma}   \geq \eps + \U_\P(x)
    \right ] \leq \frac{8x^2}{\gamma \eps^2} + \left ( \frac{16}{\eps^2} + 2 \right ) \U_\P(x).
  \end{equation}
  Analyzing the second term on the right-hand side of \eqref{eq:l1-proof-first-prob} and applying the above inequality, we see that
  \begin{align}
    \P \left [ \supkm \frac{|S_k|}{k}   \geq  \frac{\eps}{2} + 2\U_\P (m^\lambda) \right ]
    &\leq \P \left [ \sup_{k \in \NN} \frac{\lvert S_k  \rvert}{k + m}   \geq  \frac{\eps}{4} + \U_\P (m^\lambda) \right ]
                     \leq \frac{128 m^{2\lambda}}{m \eps^2} + \left ( \frac{256}{\eps^2} + 2 \right ) \U_\P(m^{\lambda}).
  \end{align}
  Consolidating terms completes the proof.
\end{proof}

We now present an analogue of \cref{theorem:l1-concentration} with an improved rate of convergence in the spirit of \citet{marcinkiewicz1937fonctions}, who studied  random variables in $L^q$ for any $q \in (1, 2)$. After taking limits, rather than yielding Kolmogorov's and Chung's SLLNs in the distribution-pointwise and uniform cases, respectively, the following result yields the SLLNs of \citet{marcinkiewicz1937fonctions} and \citet{waudby2024distribution}.

\begin{theorem}[$L^q$ concentration for the strong law of large numbers]\label{theorem:lq-concentration}
Fix $\P \in \Pcalbar$, let $q \in [1,2)$, and define the truncated $q^\tth$ absolute moment $\U_\P^\brackq(x)$ of $X_1$ under $\P$ as
  \begin{equation}
    \U_\P^\brackq(x) := \EE_\P \left [ |X_1|^q \1 \{ |X_1|^q \geq x \} \right ];\quad x\geq 0.
  \end{equation}
  For any $\eps > 0$ and $m \in \NN$, we have %
\begin{align}
  \P \left [ \supkm  \frac{|S_k|}{k^{1/q}} \geq \eps \right ] \leq \frac{4\exp \left ( -m^{1/q - 1/2} \right )}{2-q}+ \frac{451q}{(\eps^2 \wedge 1)(2-q)} \U_\P^\brackq\left(\eps^q \frac{m^{1/2 - q/4}}{38}\right). \label{eq:lq-concentration}
\end{align}
\end{theorem}
The proof of \cref{theorem:lq-concentration} can be found in \cref{proof:lq-concentration}. Notice that \cref{theorem:lq-concentration} may be interpreted as an improvement over \cref{theorem:l1-concentration} even in the case $q=1$ since the decay of the first term in the right-hand side of \eqref{eq:lq-concentration} is exponential in $m$ while that of \eqref{eq:l1-concentration} is only polynomial. Using a line of reasoning similar to the one that followed \cref{theorem:l1-concentration}, one can use \cref{theorem:lq-concentration} to deduce the SLLNs of Kolmogorov and \citet{marcinkiewicz1937fonctions}. We state their result and provide a short proof here.

\begin{corollary}[Marcinkiewicz--Zygmund strong laws of large numbers \citep{marcinkiewicz1937fonctions}]\label{corollary:mz-slln}
  Let $\P \in \Pcalbar$ and $q \in [1,2)$. If $\EE_\P|X_1|^q < \infty$, then
    $S_n/n = o(n^{1/q - 1})$ with $\P$-probability one. 
\end{corollary}
\begin{proof}
  Observe that if $\EE_\P|X_1|^q < \infty$, then for any $\eps > 0$, $\U_\P^\brackq (\eps^q m^{1/2 - q/4}/38) \tozero$ as $\mto$. Applying \cref{theorem:lq-concentration}, we have for any $\eps > 0$, $\P[\supkm k^{-1/q} |S_k| \geq \eps] \tozero$ as $\mto$, which completes the proof.
\end{proof}
In fact, the following distribution-uniform generalization of the Marcinkiewicz--Zygmund SLLN is an immediate corollary of \cref{theorem:lq-concentration} in the same way that Chung's SLLN is immediate from \cref{theorem:l1-concentration}. 
\begin{corollary}[Distribution-uniform $L^q$ SLLNs for $q \in [1, 2)$ \citep{chung_strong_1951,waudby2024distribution}]\label{corollary:wlr-slln}
Let $q \in [1,2)$ and let $\Pcal \subset \Pcalbar$ be a family of probability measures for which the random variable $X_1$ has a $\Pcal$-uniformly integrable $q^\tth$ moment:
\begin{equation}
  \lim_{\xto} \supP \EE_\P \left [ |X_1|^q \1 \{ |X_1|^q \geq x \} \right ] = 0.
\end{equation}
The SLLN holds uniformly in $\Pcal$ at a rate of $o(n^{1/q-1})$ meaning that for any $\eps > 0$,
\begin{equation}\label{eq:wsetal}
  \lim_\mto \supP \P \left [ \supkm  \frac{|S_k|}{k^{1/q}}  \geq \eps \right ] = 0.
\end{equation}
\end{corollary}
\cref{corollary:wlr-slln} follows from \cref{theorem:lq-concentration} by taking a supremum over $\Pcal$ and a limit as $\mto$ on both the left-hand and right-hand sides of the inequality \eqref{eq:lq-concentration}. Again, note that by \citet[Theorem 1]{waudby2024distribution}, uniform integrability of the $q^\tth$ moment is necessary and sufficient to conclude \eqref{eq:wsetal} so \cref{theorem:lq-concentration} is sharp in the sense that it yields pointwise and uniform Marcinkiewicz--Zygmund-type SLLNs under the same moment assumptions.

In the following section, we consider the case of $q=2$ and arrive at an iterated logarithm inequality that can be used to derive the upper bound in the law of the iterated logarithm \citep{kolmogorov1929uber,hartman1941law}.

\section{Iterated logarithm inequalities with finite variances}\label{section:lil}
While the former sections provided inequalities that can be used to deduce that $n^{-1}S_n  = o(n^{1/q-1})$ with $\P$-probability one for $q \in [1,2)$, such a deduction is not possible when $q = 2$ as exemplified by the Hartman-Wintner law of the iterated logarithm \citep{hartman1941law} which states that if $\sigma_\P^2 := \Var_\P[X_1] \in (0, \infty)$, then
\begin{equation}\label{eq:khw-lil}
  \limsup_{\ntoinfty}\frac{S_n / \sigma_\P}{\sqrt{2 n \log \log n}}  = 1\qquad\text{with $\P$-probability one.}
\end{equation}
Juxtaposing the non-asymptotic results of the previous section with the asymptotic statement in \eqref{eq:khw-lil} naturally motivates the question: \emph{Do there exist non-asymptotic time-uniform concentration inequalities for sums of random variables with iterated logarithm rates?} It appears that this question was first posed and partially solved by \citet{darling1967iterated,darling1968some}. A simplified version of \citet[Eq.~(22)]{darling1967iterated} states that if $\infseqn{X_n}$ are \iid{} and $\widebar \sigma$-sub-Gaussian, meaning that
\begin{equation}
  \EE_\P \left [ e^{t X_1 } \right ] \leq e^{ t^2 \sigmasqub / 2 };
  \qquad t \in \RR,
\end{equation}
then for any $m \in \NN \setminus \{1,2\}$ and any $\eps > 0$,  
\begin{equation}\label{eq:darling-robbins-lil}
  \P \left [ \supkm \frac{|S_k / \widebar \sigma|}{\sqrt{k (2(1+\eps)^2 \log \log k + 2(1+\eps) \log 2 )}} \geq 1 \right ] \leq \frac{\log_{1+\eps}^{-\eps}(m)}{\eps }.
\end{equation}
It is easy to check that \eqref{eq:darling-robbins-lil} implies that $\limsup_{\ntoinfty} |S_n/\widebar\sigma| / \sqrt{2  n \log \log n} \leq 1$ with $\P$-probability one, an upper bound resembling the asymptotic behaviour of \eqref{eq:khw-lil} but with $\sigma_\P^2$ replaced by the sub-Gaussian variance proxy $\widebar \sigma^2$. Note that \eqref{eq:darling-robbins-lil} cannot be used to conclude that the almost sure upper bound of 1 is always attained as in \eqref{eq:khw-lil}. The iterated logarithm inequalities that we discuss in this section are in the same spirit as \citet{darling1967iterated}, meaning that they are non-asymptotic inequalities that yield known sharp almost sure upper bounds as consequences. We do not consider corresponding lower bounds here.

Several other iterated logarithm inequalities exist --- sometimes referred to as ``finite laws of the iterated logarithm'' --- such as in \citet{darling1968some,de2004self,jamieson2014best,kaufmann2016complexity,zhao2016adaptive}, and \citet{howard_exponential_2018}. What all of these iterated logarithm inequalities including those of \citet{darling1967iterated,darling1968some} have in common, however, is that they require the underlying random variables to be sub-Gaussian or at the very least, have finite moment generating-like functions in the case of self-normalized bounds \citep{de2004self,howard2018uniform}. Meanwhile, the law of the iterated logarithm in \eqref{eq:khw-lil} is only a statement about random variables in $L^2$ that need not have any finite moment higher than a variance.
A partial exception to this rule exists in \citet{howard_exponential_2018}, where the authors derive a sub-Gaussian iterated logarithm inequality with a variance proxy taking the form of a convex combination of bounds on the variance and a self-normalizing term. We will use this inequality in our proof of \cref{theorem:lil}.
\verbose{
Nevertheless, the aforementioned iterated logarithm inequalities do not directly yield the upper bound in the LIL under only a finite second moment assumption nor do they yield distribution-uniform generalizations of the Hartman-Wintner LIL. The following inequalities will serve precisely this purpose.
}

\begin{theorem}[An $L^2$ iterated logarithm inequality]\label{theorem:lil}
  Fix $\P \in \Pcalbar$, let $\widebar \sigma \in (0, \infty)$, and assume that $X_1$ has finite variance $\Var_\P[X_1] \leq \sigmasqub$. Define
  \begin{equation}
    \widebar \U_\P^\bracktwo(x) := \EE_\P[|X_1^2 - \sigma_\P^2| \1 \{ |X_1^2 - \sigma_\P^2| \geq x \}];\ x \geq 0.
  \end{equation}
Then for any $\eps > 0$, $\lambda \in (0, 1/2)$, and $m \in \NN \setminus \{1,2\}$, we have  
\begin{align}
  &\P \left [ \supkm  \frac{ |S_k/ \widebar \sigma|}{c_\eps \sqrt{k  \left ( \log \log ( (1 + \eps )^2 k) + \ell_\eps \right )}} 
  \geq 1 \right ] \leq \frac{2^{1+\eps} \log_{1+\eps}^{-\eps}(2m/3)}{\eps \zeta(1+\eps) } + \frac{262}{(\eps^2 \widebar \sigma^4) \land 1} \left ( m^{2\lambda - 1} + \widebar \U_\P^\bracktwo(m^\lambda) \right ),
\end{align}
  where $c_\eps := ( (1+\eps)^{5/4} + (1+\eps)^{3/4} ) / \sqrt{2}$, $\ell_\eps := \log(2 \zeta(1+\eps) / \log(1 + \eps))$, and $\zeta$ is the Riemann zeta function given by $\zeta(z) := \sum_{j=1}^\infty j^{-z};\ z > 1$. 
\end{theorem}
The proof of \cref{theorem:lil} can be found in \cref{proof:lil}.
Applying \cref{theorem:l1-concentration,theorem:lil} together, we have the following studentized analogue of the above inequality with $\sigmasqub$ replaced by the sample variance.

\begin{corollary}[A studentized $L^2$ iterated logarithm inequality]\label{corollary:self-normalized-lil}
  Fix $\P \in \Pcalbar$ and assume that $X_1$ has finite variance $\sigma_\P^2 := \Var_\P[X_1] \in (0, \infty)$.
Define the truncated normalized second moment $\widehat \U_\P^\bracktwo(x)$  by 
  \begin{equation}\label{eq:studentized second moment}
    \widehat \U_\P^\bracktwo(x) := \EE_\P \left [ \frac{X_1^2}{\sigma_\P^2} \1 \left \{ \frac{X_1^2}{\sigma_\P^2} \geq x \right \} \right ];\quad x \geq 0,
  \end{equation}
  and for each $n \in \NN$ the sample variance $\widehat \sigma_n^2 := \frac{1}{n} \sum_{i=1}^n (X_i - \widehat \mu_n)^2$ where $\widehat \mu_n := \frac{1}{n} \sum_{i=1}^n X_i$.
  Then for any $\eps > 0$, $\lambda \in (0,1/2)$, and $m \in \NN \setminus\{1,2\}$,
\begin{align}
  &\P \left [ \supkm  \frac{| S_k/ \widehat \sigma_k|(1-2(\eps\land 1/4))^{1/2}}{c_\eps \sqrt{  k  \left ( \log \log ( (1 + \eps )^2 k) + \ell_\eps \right )}} 
  \geq 1 \right ] \leq \frac{2^{1+\eps}\log_{1+\eps}^{-\eps}(2m/3)}{\eps \zeta(1+\eps)} + \frac{786}{\eps^2 \land (1/16)} \left ( m^{2\lambda-1} + \widehat \U_\P^\bracktwo(m^\lambda) \right ),
\end{align}
where $c_\eps$ and $\ell_\eps$ are as in \cref{theorem:lil}.
\end{corollary}
The proof of \cref{corollary:self-normalized-lil} can be found after that of \cref{theorem:lil} in \cref{proof:lil}.
As applications of \cref{theorem:lil} and \cref{corollary:self-normalized-lil}, we have the following distribution-uniform generalizations of the upper bound in the Hartman-Wintner law of the iterated logarithm, which seem to be new to the literature.
\begin{corollary}[Distribution-uniform laws of the iterated logarithm]\label{corollary:uniform-lil}
Let $\Pcal \subset \Pcalbar$. If the second moment of $X_1$ is $\Pcal$-uniformly integrable, i.e.~$\sup_\Pin \U_\P^\bracktwo(x) \searrow 0$ as $\xto$, then there exists  $\widebar \sigma \in (0, \infty)$ so that $\sup_\Pin \Var_\P[X_1] \leq \widebar \sigma^2$ and  the law of the iterated logarithm holds uniformly in $\Pcal$, meaning that for any $\delta > 0$ we have
  \begin{equation}\label{eq:uniform-lil-variance}
    \lim_\mto \supP \P \left [ \supkm  \frac{|S_k / \sigmaub| }{\sqrt{2  k \log \log k}}  \geq 1 + \delta
    \right ] = 0.
  \end{equation}
 Furthermore, if $\Var_\P[X_1] > 0$ for each $\Pin$ and the normalized random variable $X_1 / \sigma_\P$ has a $\Pcal$-uniformly vanishing truncated normalized second moment, i.e.~$\sup_\Pin \widehat \U_\P^\bracktwo(x) \searrow 0$ as $\xto$\verbose{, where $\widehat \U_\P^\bracktwo(x)$ is defined in \eqref{eq:studentized second moment}}, then the above holds with $\widebar \sigma$ replaced by $\widehat \sigma_k$.\verbose{for any $\delta > 0$,
\begin{equation}\label{eq:uniform-lil-sn}
    \lim_\mto \supP \P \left [ \supkm  \frac{|S_k / \widehat \sigma_k |}{\sqrt{2  k \log \log k}}  \geq 1 + \delta
    \right ] = 0.
  \end{equation}
  }
\end{corollary}

Clearly, \eqref{eq:uniform-lil-variance} implies the upper bound of the LIL in \eqref{eq:khw-lil} since taking $\Pcal = \{ \P \}$ and $\widebar \sigma = \sigma_\P$ for any distribution $\P$ for which $\sigma_\P^2 < \infty$ yields for any $\delta > 0$,
\begin{equation}
  \P \left [ \limsup_\ntoinfty  \frac{|S_n / \sigma_\P|}{\sqrt{2 n \log \log n}}  \geq 1 + \delta \right ] = \lim_\mto \P \left [ \supkm \frac{|S_k / \sigma_\P|}{\sqrt{2 k \log \log k}}  \geq 1 + \delta \right ] = 0.
\end{equation}
Nevertheless, the results of this section illustrate additional details about those distributional properties to which the asymptotics of the LIL are uniform. In particular, \cref{corollary:uniform-lil} can be viewed as an extension of Chung's $L^1$ uniform SLLN discussed in \cref{corollary:chung-slln} to random variables in $L^2$.

The results of this section focused on iterated logarithm inequalities for sums of random variables in $L^2$, but LILs of different forms have been derived for random variables without finite variances. For example, \citet{klass1976toward} derives a LIL for random variables in $L^1$ where instead of a $\sqrt{n \log \log n}$-type scaling, the asymptotic behaviour of the partial sums $S_n$ is dictated by properties of truncated first and second moments. In particular, $\U_\P^\brackone(\cdot)$ is an important quantity in describing the asymptotic behavior of $S_n$ when the variance of $X_1$ is infinite. See also \citet{de2004self} for self-normalized LILs that also rely on nonnegative supermartingale concentration and resemble \cref{corollary:self-normalized-lil} in spirit.

\section{Applications to some Baum-Katz-type strong laws}\label{section:baum-katz}

Let us now observe how the concentration inequalities of the previous sections can be used to derive SLLNs and LILs in the spirit of \citet{baum1965convergence} (see also \citet{neri2025quantitative}) and in fact strengthen some of (the forward implications in) the results in \citep{baum1965convergence}.
For a probability measure $\P \in \Pcalbar$ and any $q \in [1,2)$, \citet[Theorem~2]{baum1965convergence} show that 
\begin{equation}\label{eq:baum-katz}
  \EE_\P[|X_1|^q \log( |X_1| + 1 )] < \infty \quad \text{if and only if}\quad \forall \eps > 0,~~ \sum_{m=1}^\infty \frac{1}{m} \P \left [ \supkm  \frac{|S_k|}{k^{1/q}}  \geq \eps \right ] < \infty.
\end{equation}
The series being finite allows one to conclude that the probability in the summand vanishes at a sufficiently fast rate $r_m$ so that $r_m / m$ is summable, while the SLLNs of Kolmogorov and \citet{marcinkiewicz1937fonctions} provide no such rate of convergence. However, note that the Baum-Katz SLLN described above is neither stronger nor weaker than the SLLNs of Kolmogorov, \citet{marcinkiewicz1937fonctions}, \citet{chung_strong_1951}, or \citet{waudby2024distribution}.
Nevertheless, the concentration inequality in \cref{theorem:lq-concentration} is sufficiently sharp to provide an explicit upper bound on the infinite series in \eqref{eq:baum-katz}.
\verbose{culminating in a strengthening and alternative proof of their result.}
\begin{proposition}[A Baum-Katz-type strong law of large numbers]\label{corollary:baum-katz}
  Let $\P \in \Pcalbar$ and $q \in [1,2)$. For any $m \in \NN$ and $\eps > 0$, define $P_m^\brackeps := \P [ \supkm |k^{-{1/q}} S_k| \geq \eps ]$. Then,
  \begin{equation}\label{eq:baum-katz-type-upperbound}
    \sum_{m=1}^\infty \frac{P_m^\brackeps}{m} \leq 1 +  \frac{c_q}{e \log(2^{1/q-1/2})} + \frac{2603q \left ( \EE_\P[|X_1|^q \log (38|X_1|^q / \eps^q + 1)] \right ) }{(2-q)^2 (\eps^2 \land 1)},
  \end{equation}
  where $c_q := 4/(2-q)$.
  \verbose{In particular, if $\Pcal \subset \Pcalbar$ is a collection of probability measures for which
  \begin{equation}
    \sup_{\Pin}\EE_\P [|X_1|^q \log (|X_1|^q + 1)] < \infty, 
  \end{equation}
  then $\sup_{\Pin}\sum_{m=1}^\infty P_m^\brackeps / m < \infty$ for any $\eps > 0$. }
\end{proposition}

Note that \cref{corollary:baum-katz} can be viewed as a strengthening of the forward implication in \eqref{eq:baum-katz} since an explicit upper bound on the series $\sum_{m=1}^\infty P_m^\brackeps / m$ is provided. The proof is short so we provide it here. 
\begin{proof}[Proof of \cref{corollary:baum-katz}]
  First, note that
  \begin{align}
    \sum_{m=1}^\infty \frac{P_m^\brackeps}{m} &= \sum_{j=1}^\infty \sum_{m=2^{j-1}}^{2^j-1} \frac{P_{m}^\brackeps}{m} \leq \sum_{j=1}^\infty \frac{\cancel{2^{j-1}}}{\cancel{2^{j-1}}} \P \left [ \sup_{k \geq 2^{j-1}}  \frac{ |S_k|}{k^{1/q}}  \geq \eps \right ] \leq 1 + \sum_{j=1}^\infty P_{2^j}^\brackeps.
  \end{align}
  Applying \cref{theorem:lq-concentration} yields
  \begin{align}
    \sum_{j=1}^\infty P_{2^j}^\brackeps \leq c_{q} \sum_{j=1}^\infty \exp \left ( -2^{j(1/q - 1/2)} \right ) + \frac{451q}{(\eps^2 \land 1)(2-q)} \sum_{j=1}^\infty  \EE_\P \left [ |X_1|^q \1 \left \{ |X_1|^q \geq \frac{\eps^q 2^{j(1/2 - q/4)}}{38} \right \} \right ].
  \end{align}
  Analyzing the series in the second term of the right-hand side, we have 
  \begin{align}
    \sum_{j=1}^\infty  \EE_\P \left [ |X_1|^q \1\left \{ |X_1|^q \geq \frac{\eps^q 2^{j(1/2 - q/4)}}{38} \right \} \right ] &= \EE_\P \left [|X_1|^q \sum_{j=1}^\infty \1 \{ 38|X_1|^q \eps^{-q} \geq 2^{j(1/2 - q/4)} \} \right ] \\
    &\leq \frac{1}{\log(2)} \EE_\P \left [ |X_1|^q \frac{
    \log(38 |X_1|^q \eps^{-q} + 1)}{1/2 - q/4} \right ],
  \end{align}
  where the inequality upper bounds the number of non-zero indicators in the series.
  Since $\log (2) \geq 0.69314$, 
  \begin{equation}
     \sum_{m=1}^\infty \frac{P_m^\brackeps}{m} \leq 1 + c_q \sum_{j=1}^\infty \exp \left ( -2^{j(1/q - 1/2)} \right ) + \frac{2603q \left ( \EE_\P[|X_1|^q \log (38|X_1|^q \eps^{-q} + 1)] \right ) }{(2-q)^2 (\eps^2 \land 1)}.
  \end{equation}
    Now, observe that
    \begin{align}
      \sum_{j=1}^\infty \exp \left ( -2^{j(1/q-1/2)} \right ) &\leq \int_{0}^\infty \exp \left ( -2^{y(1/q-1/2)} \right )\dd y = \frac{1}{\log(2^{1/q - 1/2})}\int_1^\infty \frac{e^{-x}}{x} \dd x \leq \frac{1}{e \log(2^{1/q - 1/2})},
    \end{align}
    where the equality uses the change of variables $x = 2^{y(1/q - 1/2)}$ and the last inequality bounds $x^{-1}$ by $1$ inside the integral. This completes the proof of the upper bound in \eqref{eq:baum-katz-type-upperbound}.
\end{proof}

Let us now consider an analogous setup to \cref{corollary:baum-katz} but for the LIL. Recall that in \citet[Theorem~6]{baum1965convergence}, the authors show that for a random variable $X_1$ with unit variance, if $\EE_\P[|X_1|^2 \log^{1+\delta} (|X_1| + 1)] < \infty$ for some $\delta > 0$, then
\begin{equation}\label{eq:baum-katz-lil}
 \text{for any } \gamma > 0,\quad \sum_{m=3}^\infty \frac{1}{m \log(m)} \P \left [ \supkm \frac{|S_k|}{\sqrt{2 k \log \log k}}  \geq 1 + \gamma \right ] < \infty.
\end{equation}
The following result employs \cref{theorem:lil} to obtain a distribution-uniform Baum-Katz-type LIL under weaker moment conditions than those listed above.
\begin{proposition}[A Baum-Katz-type law of the iterated logarithm]\label{proposition:baum-katz-lil}
  Fix $\P \in \Pcalbar$ with $\Var_\P[X_1] = 1$. For $\eps > 0$, denote
  \begin{equation}
    P_m^\brackeps := \P \left [ \supkm  \frac{|S_k|}{c_\eps \sqrt{k [\log \log ((1 + \eps)^2 k ) + \ell_\eps]}}  \geq 1 \right ],
  \end{equation}
  where the constants $c_\eps$ and $\ell_\eps$ are as in \cref{theorem:lil}.
  Then for any $\eps  > 0$, and any $\delta > 0$,
  \begin{align} \label{eq:251009}
    \sum_{m=3}^\infty \frac{P_m^\brackeps}{m \log (m)} \leq \sum_{m=3}^\infty \frac{2^{1+\eps} \log^\eps(1 + \eps) \eps^{-1} \zeta(1 + \eps)}{m \log^{1 + \eps}(2m/3)} + \frac{262}{\eps^2 \land 1} \left ( \sum_{m=3}^\infty \frac{m^{-4/3}}{ \log (m)} + \sum_{m=3}^\infty \frac{ \EE_\P[X_1^2 \log^\delta (X_1^2 + 1)]}{(1/3)^{\delta}m \log^{1+\delta}(m)}  \right ).
  \end{align}
  In particular, if $\Pcal \subset \Pcalbar$ is a collection of probability measures for which $X_1$ has unit variance and $\sup_\Pin \EE_\P[X_1^2 \log^\delta (X_1^2 + 1)] < \infty$ for some $\delta > 0$, then
  \begin{equation}
    \sup_\Pin \sum_{m=3}^\infty \frac{P_m^\brackeps}{m \log (m)} < \infty.
  \end{equation}
\end{proposition}
\begin{proof}
  Applying \cref{theorem:lil} with $\lambda = 1/3$, we have 
  \begin{align}
    \sum_{m=3}^\infty \frac{1}{m\log m}P_m^\brackeps &\leq \sum_{m=3}^\infty \frac{1}{m\log m} \left ( \frac{2^{1+\eps}\log^{-\eps}_{1+\eps}(2m/3)}{\eps \zeta(1+\eps)  } + \frac{262}{\eps^2 \land 1 } \left ( m^{-1/3} + \EE_\P[|X_1^2 - 1| \1 \{ |X_1^2-1| > m^{1/3} \}] \right )\right )\\
    &\leq \sum_{m=3}^\infty \frac{2^{1+\eps}\log_{1+\eps}^{-\eps}(2m/3)}{m \log(m)  \eps \zeta(1+\eps)} + \frac{262}{\eps^2 \land 1}  \sum_{m=3}^\infty  \left(\frac{m^{-4/3}}{ \log(m)} + \frac{\EE_\P \left [ X_1^2 \1 \{ X_1^2 + 1 > m^{1/3} \} \right ]}{m \log m} \right ).
  \end{align}
  Now, the result follows by observing that for any $a,b \geq 1$, we have $\1 \{ a > b \} \leq \log^\delta(a) / \log^\delta(b)$.
\end{proof}

Even in the case where $\Pcal = \{ \P \}$ is taken to be a singleton, \cref{proposition:baum-katz-lil} improves on \citep[Theorem~6]{baum1965convergence} by only requiring that $\EE_\P[X_1^2 \log^\delta(X_1^2 + 1)] < \infty$ for some $\delta > 0$ rather than for some $\delta > 1$.

\section{Pathwise strong laws and laws of the iterated logarithm}\label{section:game-theoretic}
While SLLNs and LILs are typically written in terms of probability-one events such as in \eqref{eq:intro-kolmogorov-slln} and \eqref{eq:khw-lil}, there has been renewed interest in \emph{pathwise} (or \emph{game-theoretic}) presentations and proofs of almost sure limit theorems that rely on the explicit construction of so-called \emph{$e$-processes}, the definition of which we review now.

\begin{definition}[$e$-process]
  \label{definition:e-process}
Fix $\P \in \Pcalbar$ and let $(\Fcal_n)_{n \in \NN_0}$ be a filtration. A nonnegative $(\Fcal_n)_{n \in \NN_0}$-adapted stochastic process $\infseqnz{E_n}$ is said to be a \emph{$\P$-$e$-process} if
\begin{equation}\label{eq:e-proc}
  \EE_\P[E_\tau] \leq 1
\end{equation}
for an arbitrary $(\Fcal_n)_{n \in \NN_0}$-stopping time $\tau$, with the convention that $E_\infty = \limsup_n E_n$. The above process is said to be a $\Pcal$-$e$-process for an arbitrary family of probability measures $\Pcal$   if \eqref{eq:e-proc} holds for all $\Pin$.
\end{definition}
Broadly speaking, given an event $A \in \Fcal_\infty$, a proof of the claim ``$\P[A] = 0$'' is often given the description of \emph{pathwise} or \emph{game-theoretic} if one constructs an explicit $\P$-$e$-process $\infseqnz{E_n}$ with the property that this process diverges pathwise on $A$, meaning that $E_n(\omega) \toinfty$ for every $\omega \in A$; see e.g.~\citet{sasai2019erdos} and \citet{ruf2022composite}.
Such a construction is directly connected to the notion of $A$ having probability zero as illustrated by Ville's theorem \citep{ville1939etude} which states that $\P[A] = 0$ if and only if
there exists a $\P$-$e$-process that diverges pathwise on $A$. 
In Ville's writing of his theorem, the $e$-process was given the interpretation of the accumulated wealth of a hypothetical gambler playing a ``fair'' sequential game. \verbose{Intuitively, a gambler playing such a game over time can never become infinitely rich except with zero probability; formally, $\P[\sup_{n \in \NN} E_n < \infty] = 1$, a consequence of Ville's inequality for nonnegative supermartingales \citet{ville1939etude} (see also \citet[\S 6.1]{howard_exponential_2018} for an elementary proof) applied to the Snell envelope of $(E_n)_{n \in \NN}$ under $\P$.}
It is because of this hypothetical gambler and the game they are playing that such proofs and constructions are often described as ``game-theoretic''.
However, the same phrase is also used to describe theorems and proofs in the so-called game-theoretic formalism of probability as set out by \citet{shafer2005probability,shafer2019game}, where Kolmogorov's axioms of measure-theoretic probability are eschewed. To emphasize that we are operating in a purely measure-theoretic setting, we drop the term ``game-theoretic'' altogether going forward and use the term ``pathwise'' instead. 
\begin{remark}\label{remark:ville-theorem-eprocess}
  One will typically find Ville's theorem stated in terms of a nonnegative $\P$-\emph{martingale} diverging to $\infty$ pathwise on an event $A$ rather than a $\P$-$e$-process doing so, but the former can be replaced by the latter without loss of generality; see \citet[Remark 3.2]{ruf2022composite}. Note that all nonnegative $\P$-martingales started at one are $\P$-$e$-processes --- a consequence of Doob's optional stopping theorem --- but there exist $e$-processes that are neither supermartingales nor martingales.
\end{remark}

\subsection{Deriving pathwise SLLNs and LILs from concentration inequalities} \label{SS:5.1}

Let us now illustrate how pathwise proofs of SLLNs and LILs can be directly derived once provided access to the concentration inequalities of \cref{section:slln,section:lil}. First, fix $q \in [1,2)$ and consider a probability measure $\P \in \Pcalbar$ so that $\EE_\P|X_1|^q < \infty$. Consider also the event $A_\qdiv := \{ S_n / n^{1/q} \text{ does not converge to 0} \}$ which states that the SLLN does not hold at the Marcinkiewicz--Zygmund rate of $o(n^{1/q-1})$.
Construct the process $\infseqnz{E_n^\brackq}$  by
\begin{equation}\label{eq:e-process-qdiv}
  E_n^\brackq := \sum_{j \in \NN} \1 \left \{\text{there exists $k\in\{m_j, \ldots, n\}$ such that }  \frac{|S_k|}{k^{1/q}} \geq \frac{1}{j}  \right \}.
\end{equation}
Here, we use the notation
\[
    m_j := \min \left \{ m \in \NN :  \frac{4\exp (-m^{1/q - 1/2})}{2-q} + \frac{451q j^2}{2-q}  \U_\P^\brackq\left(\frac{m^{1/2 - q/4}}{38 j^q}\right) \leq 2^{-j} \right \}; \qquad j \in \NN,
\]
where $\U_\P^\brackq$ is as in \cref{theorem:lq-concentration}.
To see why $E_n^\brackq$ forms a $\P$-$e$-process that diverges pathwise on $A_\qdiv$, first note that for any stopping time $\tau$, we have 
\begin{equation}
  \EE_\P \left [ E_\tau^\brackq \right ] \leq \sum_{j \in \NN} \P \left [ \bigcup_{k \geq m_j} \left\{ \frac{|S_k|}{k^{1/q}} \geq \frac{1}{j} \right\} \right ] \leq \sum_{j\in \NN} 2^{-j}, 
\end{equation}
where the second inequality follows from \cref{theorem:lq-concentration} instantiated with $\eps = 1/j$. It follows that $E_n^\brackq$ forms a $\P$-$e$-process. Let us now see why $E_n(\omega) \nearrow \infty$ as $\ntoinfty$ for every $\omega \in A_\qdiv$. Notice that by definition of $A_\qdiv$, for every $\omega \in A_\qdiv$ there exists some $T(\omega) \in \NN$ so that for every $j \geq T(\omega)$, we have $|S_k(\omega)|/k^{1/q} \geq 1/j$ for infinitely many $k \in \NN$. Therefore,  the $j^\tth$ summand in \eqref{eq:e-process-qdiv} is equal to one eventually, and therefore $E_n^\brackq(\omega)$ diverges as $\ntoinfty$.

A similar story can be told for the LIL.
Consider a probability measure $\P \in \Pcalbar$ for which $\sigma_\P^2 := \Var_\P[X_1] \in (0, \infty)$ and define the event $A_\fluc$ for which the LIL upper bound does not hold by
\begin{equation}
  A_\fluc := \left \{ \limsup_\ntoinfty  \frac{|S_n/ \sigma_\P|}{\sqrt{2 n \log \log n}}  > 1 \right \},
\end{equation}
and the process $\infseqnz{E_n^\bracktwo}$ by
\begin{equation}
  E_n^\bracktwo := \sum_{j\in\NN} \1 \left \{\text{there exists $k\in\{m_j, \ldots, n\}$ such that } \frac{|S_k / \sigma_\P|}{c_{1/j} \sqrt{k [ \log \log ((1 + 1/j)^2 k) + \ell_{1/j}]}}  \geq 1  \right \},
\end{equation}
and $E_0^\bracktwo := 0$
where $m_j$ is the smallest integer for which the right-hand side of the inequality in \cref{theorem:lil} instantiated with $(\lambda, \eps, \widebar \sigma^2) = (1/3, 1/j, \sigma_\P^2)$ is at most $2^{-j}$, i.e.
\begin{equation}
  m_j := \min \left \{ m \in \NN \setminus \{1,2\} : \frac{2^{1+1/j}j\log^{-1/j}_{1+1/j}(2m/3)}{\zeta(1 + 1/j)} + \frac{262}{(\sigma_\P^4 / j^2)\land 1} \left ( m^{-1/3} + \widebar \U_\P^\bracktwo (m^{1/3}) \right ) \leq 2^{-j} \right \},
\end{equation}
where $c_{1/j}$, $\ell_{1/j}$, $\zeta$, and $\widebar \U_\P$ are as in \cref{theorem:lil}.
The justification for why $\infseqnz{E_n^\bracktwo}$ is both a $\P$-$e$-process and diverges pathwise on $A_\fluc$ is essentially the same as above but with \cref{theorem:lil} invoked  instead of \cref{theorem:lq-concentration}. 

As far as we know, $\infseqnz{E_n^\brackq}$ and $\infseqnz{E_n^\bracktwo}$ are the first $e$-processes to be derived for the SLLN with Marcinkiewicz--Zygmund rates and for the LIL only under finite $q^\tth$ moment assumptions when $q \in (1,2]$. The case of $q=1$ under a finite first moment assumption was completed in \citet[Theorem~4.3]{ruf2022composite} and the case of the LIL was studied in \citet{sasai2019erdos} for self-normalized martingales.

\subsection{A distribution-uniform analogue of Ville's theorem for event lattices}

The discussion thus far has focused on $\P$-$e$-processes for a single probability measure $\P$. We will provide an analogue of Ville's theorem for a family of probability measures $\Pcal$ when applied to events that can be represented as union-intersections of certain lattices of events. Concretely, consider events $A$ that can be written as
\begin{equation}
  A = \bigcup_{\eps > 0} \bigcap_{m=1}^\infty A^\brackmeps
\end{equation}
for some collection $(A^\brackmeps)_{(m,\eps) \in \NN \times \RR^+}$ with the property that $A^{(m_1, \eps_1)} \supseteq A^{(m_2, \eps_2)}$ for every $(m_1, \eps_1), (m_2, \eps_2) \in \NN\times \RR^+$ whenever $m_1 \leq m_2$ and $\eps_1 \leq \eps_2$. The lattice structure is induced from the aforementioned set inclusion. For brevity, we will refer to such collections as ``event lattices''. 
The reason for considering events of this type is because distribution-uniform SLLNs and LILs are implicitly statements about the lattices that represent events rather than the events \emph{per se}. For example, recall that Chung's SLLN (\cref{corollary:chung-slln}) states that if the first moment is $\Pcal$-uniformly integrable, then
\begin{equation}\label{eq:chung-again}
  \forall \eps > 0,\quad \lim_\mto \sup_\Pin \P \left [ \supkm  \frac{|S_k|}{k}  \geq \eps \right ] = 0.
\end{equation}
Note that $A_\onediv$ can be written as a union-intersection of the events in the above probabilities, i.e.
\begin{equation}
  A_\onediv = \left \{ \frac{S_n}{n} \text{ does not converge to 0} \right \} = \bigcup_{\eps > 0} \bigcap_{m=1} \left \{ \supkm \frac{|S_k|}{k}  \geq \eps \right \}.
\end{equation}
Similarly, the distribution-uniform SLLN of \citet{waudby2024distribution} (\cref{corollary:wlr-slln}) is a statement about the lattice given by $A^\brackmeps := \{\supkm |S_k| / k^{1/q}   \geq \eps\}$, and the distribution-uniform LIL upper bound in \cref{corollary:uniform-lil} is one about the lattice given by $A^\brackmeps := \{ \supkm|S_k/\sigma| / \sqrt{2 \widebar  k \log \log k} \geq 1 + \eps \}$ where $\widebar \sigma$ is defined in the statement of the corollary.
Relevant to the present section, we will soon see that the event lattices used to represent ``distribution-uniform convergence'' in the sense of \citet{chung_strong_1951} and as seen in \eqref{eq:chung-again} are crucial to a uniform generalization of $e$-processes diverging to $\infty$, and that these two notions of uniformity are equivalent in a certain sense. To this end, consider the following definition.

\begin{definition}[Pathwise tail-uniformity]
  \label{definition:tail event-uniformity}
   Let $\infseqnz{E_n}$ be a process. We say that $\infseqnz{E_n}$ diverges pathwise to $\infty$ \emph{tail-uniformly} on an event lattice $(A^\brackmeps)_{(m,\eps)\in \NN \times \RR^+}$ if for all $\eps > 0$,
  \begin{equation}
   \lim_\mto \inf_{\omega \in A^{(m, \eps)}} \sup_{n \in \NN} E_n(\omega) = \infty.%
 \end{equation}
\end{definition}
\cref{definition:tail event-uniformity} rules out those processes that may diverge pathwise to $\infty$ for all $\omega \in A^\brackmeps$ for every $m \in \NN$ and $\eps > 0$ but may not do so uniformly in this sequence of tail events as $\mto$. Concretely, if a process does not diverge pathwise tail-uniformly, then there could exist some $\eps>0$ and some constant $U$ so that no matter what value $m \in \NN$ is taken to be, there may exist some $\omega_{m,U} \in A^\brackmeps$ depending on $m$ and $U$ for which $\sup_{n \in \NN} E_n(\omega_{m,U}) \leq U$. 
In \cref{section:implications-composite}, we give an example of an $e$-process that diverges pathwise on the event that the SLLN fails to hold, but not tail-uniformly on a natural lattice that approximates that event.

The following result demonstrates that \cref{definition:tail event-uniformity} can be used to equivalently characterize strong asymptotic events with distribution-uniform probability zero in the sense of \citet{chung_strong_1951}.
\begin{theorem}[A distribution-uniform analogue of Ville's theorem for event lattices]\label{P:251021}\label{theorem:distribution-uniform-ville}
  Fix a family of probability measures $\Pcal$, a filtration $(\Fcal_n)_{n \in \NN_0}$, and the event lattice consisting of unions of events: 
  \[\Acal := \left ( \bigcup_{k\geq m} A_k^\brackeps \right )_{(m , \eps) \in \NN \times \RR^+},\]
  where $(A_k^\brackeps)_{(k,\eps) \in \NN \times \RR^+}$ is a collection of sets satisfying $A_k^\brackeps \in \Fcal_k$ for all $(k,\eps) \in \NN \times \RR^+$.
  Then
  \begin{equation}\label{eq:uniform-convergence-abstract}
    \forall \eps > 0,\quad \lim_\mto \supP \P \left [ \bigcup_{k \geq m} A_k^\brackeps \right ] = 0
  \end{equation}
  if and only if there exists an $(\Fcal_n)_{n \in \NN_0}$-adapted $\Pcal$-$e$-process $\infseqnz{E_n}$ that diverges pathwise to $\infty$ tail-uniformly on $\Acal$.
\end{theorem}
Note that \cref{theorem:distribution-uniform-ville} can be stated irrespective of any set $A$ that uses $\Acal$ as a representing event lattice.
The proof of \cref{theorem:distribution-uniform-ville} is short and constructive so we present it here.
\begin{proof}[Proof of \cref{P:251021}]
  Suppose that \eqref{eq:uniform-convergence-abstract} holds.
  For any $j \in \NN$, define
  \begin{equation}
    m_j := \min \left \{ m \in \NN : \supP \P \left [ \bigcup_{k \geq m} A_k^{(1/j)} \right ] \leq 2^{-j} \right \},
  \end{equation}
  noting that $m_j$ is non-decreasing in $j$. For any $\omega \in \Omega$ and $n \in \NN_0$, define
  \begin{equation}
    E_n(\omega) := \sum_{j \in \NN} \1 \left \{ \omega \in \bigcup_{k = m_j}^n A_k^{(1/j)} \right \}.
  \end{equation}
  To show that $\infseqnz{E_n}$ is a $\Pcal$-$e$-process, let $\tau$ be any $(\Fcal_n)_{n \in \NN_0}$-stopping time and observe that
  \begin{align}
    \EE_\P \left [ E_\tau \right ] &\leq \sum_{j \in \NN} \P \left [ \left(\bigcup_{k = m_j}^\tau A_k^{(1/j)}\right) \cap \{ \tau \geq m_j \} \right ] \leq \sum_{j \in \NN} \P \left [ \bigcup_{k \geq m_j} A_k^{(1/j)} \right ] \leq 1
  \end{align}
  by construction of $m_j$.
  We next argue that $(E_n)_{n \in \NN_0}$ diverges pathwise to $\infty$ tail-uniformly on $\Acal$. To this end, let $\eps > 0$ and $U \in \NN$. Define $j(\eps) := \left \lceil 1/\eps \right \rceil$ and consider an arbitrary $m \geq m_{j(\eps) + U}$. Then for any $\omega \in A^{(m, \eps)} =  \bigcup_{k\geq m} A_k$, we have 
  \begin{align}
    \sup_{n \in \NN} E_n(\omega) &\geq \sum_{j=j(\eps)}^{j(\eps) + U} \1 \left \{ \omega \in  \bigcup_{k = m_j}^\infty A_k^{(1/j)} \right \} \geq \sum_{j=j(\eps)}^{j(\eps) + U} \1 \left \{ \omega \in  \bigcup_{k = m_j}^\infty A_k^{(\eps)} \right \} = U + 1,
  \end{align}
  where the second inequality follows from monotonicity of the event lattice and the final equality follows from the fact that $\bigcup_{k \geq m} A_k^\brackeps \subseteq \bigcup_{k \geq m_j} A_k^\brackeps$ for every $j \leq j(\eps) + U$. Since $U$ was arbitrary, we have 
  \begin{equation}
    \lim_\mto \inf_{\omega \in A^{(m,\eps)}} \sup_{n \in \NN} E_n(\omega) = \infty.
  \end{equation}
  
We now prove the converse. Suppose that $\infseqnz{E_n}$ diverges pathwise to $\infty$  tail-uniformly on $\Acal$. By \cref{definition:tail event-uniformity}, we have that for every $\eps> 0$ and $U \in \NN$, there exists some $m_U$ for which $\sup_{n \in \NN}E_n(\omega) \geq U + 1$ for all $\omega \in A^{(m_U,\eps)}$. %
Considering the stopping time $\tau := \min \{n \in \NN_0: E_n \geq U\}$ we get
\begin{align}
  \sup_\Pin \P \left [ \bigcup_{k = m_U}^\infty A_k^\brackeps \right ] &\leq \sup_\Pin \P \left [ \sup_{n \in \NN} E_n \geq U + 1 \right ] \leq \supP \P \left [ E_\tau \geq U \right ] \leq \frac{1}{U} \sup_\Pin \EE_\P \left [ E_\tau \right ] \leq \frac{1}{U}.
\end{align}
Since $U$ was arbitrary, this completes the proof.
\end{proof}
With \cref{theorem:distribution-uniform-ville} in place, it can be used to provide necessary and sufficient conditions for a uniform pathwise SLLN to hold. Fix $q \in [1, 2)$ and let $A_\qdiv$ be the event that the SLLN does not hold at the Marcinkiewicz--Zygmund rate of $o(n^{1/q - 1})$. This event has an event lattice presentation as follows: 
\begin{equation}
  A_\qdiv = \left \{ \lim_\ntoinfty \frac{S_n}{n^{1/q}}  \neq 0 \right \} \equiv \bigcup_{\eps > 0} \bigcap_{m=1}^\infty \bigcup_{k=m}^\infty \left \{  \frac{|S_k|}{k^{1/q}}   \geq \eps \right \}.\label{eq:Aqdiv as a union of nested limsups}
\end{equation}
We now have the following near-immediate corollary.
\begin{corollary}[Uniform pathwise $L^q$ strong laws of large numbers]\label{corollary:game-theoretic sllns}
  Let $q \in [1, 2)$ and $\Pcal \subset \Pcalbar$. Let $(\Fcal_n)_{n \in \NN_0}$ be the filtration generated by $\infseqn{X_n}$ and define the event lattice
  \begin{equation}
    \Acal_\qdiv := \left ( \bigcup_{k\geq m}\left \{  \frac{|S_k|}{k^{1/q}}    \geq \eps \right \} \right )_{(m, \eps) \in \NN \times \RR^+}.
  \end{equation}
  Then the following three conditions are equivalent:
  \begin{enumerate}[label = (\roman*)]
  \item $X_1$ has a uniformly integrable $q^\tth$ moment:
    \begin{equation}
      \lim_\xto \supP \EE_\P \left [ |X_1|^q \1 \{ |X_1|^q \geq x \} \right ] = 0.
    \end{equation}
  \item The SLLN holds uniformly in $\Pcal$ at a rate of $o(n^{1/q - 1})$, meaning that
    \begin{equation}
      \forall \eps > 0,\quad \lim_\mto \supP \P \left [ \supkm  \frac{|S_k|}{k^{1/q}}   \geq \eps \right ] = 0.
    \end{equation}
    \item There exists a $\Pcal$-$e$-process that diverges pathwise to $\infty$ tail-uniformly on $\Acal_\qdiv$.
  \end{enumerate}
\end{corollary}
\begin{proof}
  The first equivalence $(i) \iff (ii)$ follows from \citet[Theorem~1]{waudby2024distribution} and the second $(ii)\iff(iii)$ follows from \cref{theorem:distribution-uniform-ville}.
\end{proof}
Let us now move on to the case of $q=2$, where \citet{ruf2022composite} prompted the future direction of \emph{``[extending] game-theoretic constructions for the law of the iterated logarithm to [uniform] settings''}. The following corollary provides one answer to this inquiry.
\begin{corollary}[Uniform pathwise $L^2$ laws of the iterated logarithm]\label{corollary:game-theoretic lils}
  Let $\Pcal \subset \Pcalbar$ be a family of probability measures for which $0 < \Var_\P[X_1] \leq \widebar \sigma^2 < \infty$ for every $\Pin$. Let $(\Fcal_n)_{n \in \NN_0}$ be the filtration generated by $\infseqn{X_n}$ and define the event lattice $\Acal_\fluc$ describing super-iterated-logarithm fluctuations:
  \begin{equation}
    \Acal_\fluc := \left ( \bigcup_{k\geq m} \left\{  \frac{|S_k/\widebar \sigma|}{\sqrt{2 k \log \log k} } \geq 1 + \eps
    \right\} \right )_{(m, \eps) \in \NN \times \RR^+},
  \end{equation}
  recognizing that $A_\fluc := \bigcup_{\eps > 0} \bigcap_{m\geq 1} \bigcup_{k\geq m}\{ |S_k /\widebar \sigma|/ \sqrt{2  k \log \log k}  \geq 1 + \eps \}$ is the converse of the upper bound in the LIL.
  Then 
    \begin{equation}\label{eq:game-theoretic-lil-upperbound}
      \forall \eps > 0,\quad \lim_\mto \sup_\Pin \P \left [ \supkm \frac{|S_k / \widebar \sigma|}{\sqrt{2 k \log \log k}}  \geq 1 + \eps \right ] = 0
    \end{equation}
    if and only if there exists a $\Pcal$-$e$-process that diverges pathwise tail-uniformly on $\Acal_\fluc$.
\end{corollary}
The above follows immediately from \cref{theorem:distribution-uniform-ville}. While there do exist (non-uniform) pathwise LILs in the literature such as those of \citet{sasai2019erdos}, they require more than 2 moments in the \iid{} case and hence are not viewed as pathwise counterparts of the upper bound in the Hartman-Wintner LIL.

%% file: proofs.tex
\section{Proof of \cref*{theorem:lq-concentration}}\label{proof:lq-concentration}

\verbose{The proof of \cref{theorem:lq-concentration} proceeds as follows. We begin by decomposing the partial sum $\sum_{i=1}^k  X_i$ into three parts. The first is a partial sum consisting of upper-truncated versions of the $X_i$'s where the common truncation level scales with $\eps m^{\lambda / q}$ for some $\lambda \in (0,1 - q/2)$. Note that we will later set $\lambda$ to $1/2 - q/4$ to arrive at the statement of \cref{theorem:lq-concentration}. The second consists of a similar partial sum but upper- and lower-truncated to the intervals $(\eps m^{\lambda / q}, (i-1)^{1/q})$ --- where the interval is taken to be the empty set if $\eps m^{\lambda / q} > (i-1)^{1/q}$ --- in particular noting that these intervals now depend on the indices $i \in \{m, m+1, \dots \}$. The third is a partial sum consisting of the remaining parts of the $X_i$'s after being truncated in the first two terms. Recall that our goal is to control $S_k$ time-uniformly with high probability, and we do so for the first term by exploiting the fact that the applied truncation induces sub-Gaussianity of the partial sums. The second and third terms are controlled by exploiting properties of the truncated random variables combined with Kolmogorov's inequality and summation by parts. Of particular note is that we will not simply apply Kronecker's lemma on a set of probability one (a common technique in classical proofs of SLLNs), but our use of summation by parts implicitly plays an analogous role in a non-asymptotic manner.}

\begin{proof}[Proof of \cref{theorem:lq-concentration}]
  Fix $m \in \NN$, $\eps > 0$, and $\lambda \in (0,1-q/2)$. We begin by considering the following three-term decomposition of $S_k$ and applying the triangle inequality to obtain
\begin{equation}
  \frac{1}{k^{1/q}} \left \lvert S_k \right \rvert \leq \frac{1}{k^{1/q}} \left \lvert \sum_{i=1}^k Y_i \right \rvert + \frac{1}{k^{1/q}} \left \lvert \sum_{i=1}^k Z_i \right \rvert + \frac{1}{k^{1/q}} \left \lvert \sum_{i=1}^k R_i \right \rvert, 
\end{equation}
where 
\begin{align}
  Y_i &:= X_i \1 \{ |X_i| \leq \eps m^{\lambda/q} \} - \EE_\P \left [ X_i \1 \{ |X_i| \leq \eps m^{\lambda/q} \} \right ],\\
  Z_i &:= X_i \1 \{ \eps m^{\lambda/q} < |X_i| \leq (i-1)^{1/q} \} - \EE_\P \left [  X_i \1 \{ \eps m^{\lambda/q} < |X_i| \leq (i-1)^{1/q} \} \right ],\\
  R_i &:= X_i \1 \{ |X_i| > \eps m^{\lambda/q} \lor (i-1)^{1/q} \} - \EE_\P \left [ X_i \1 \{ |X_i| > \eps m^{\lambda/q} \lor (i-1)^{1/q} \} \right ].
\end{align}
In the following three steps, we will derive time-uniform concentration inequalities for the partial sums of $\infseqn{Y_n}$, $\infseqn{Z_n}$, and $\infseqn{R_n}$, respectively, ultimately combining them in the fourth step to obtain the desired result.

\subsubsection*{Step I: Time-uniform concentration of $\sum_{i=1}^k Y_i$.}
Observe that by truncating at $\eps m^{\lambda/q}$, the random variable $Y_i$ is supported on a finite interval $[a,b]$ with $b-a = 2 \eps m^{\lambda /q}$ and hence $Y_i$ is sub-Gaussian with variance proxy $(b-a)^2 / 4 = \eps^2 m^{2\lambda /q}$ for each $i \in \NN$. Moreover, these summands are all independent of each other. %
Exploiting this sub-Gaussianity, we employ a concentration inequality due to \citet[Theorem~1]{howard2018uniform}, which implies that for any $\alpha \in (0, 1)$, and any nondecreasing function $h : [0, \infty) \to (0, \infty)$, it holds that 
\begin{equation}\label{eq:howard-stitched}
  \P \left [ \exists k \in \NN : \left \lvert \sum_{i=1}^k Y_i \right \rvert \geq \eps m^{\lambda/q} \sqrt{c^2 k \left ( \log h \left ( \log(k) \right ) + \log(1/\alpha) \right )} \right ] \leq 2 \alpha \sum_{j=0}^\infty \frac{1}{h(j)},
\end{equation}
where $c :=  \left ( e^{1/4} + e^{-1/4} \right ) / \sqrt{2}$.
Now, define
\begin{equation}
  h(x) := \exp \left ( \left ( e^x + m \right )^{2(1-\lambda) /q - 1} \right ); \qquad x \geq 0. %
\end{equation}
Take $\alpha := \exp ( -m^{2(1-\lambda)/q - 1} )$
 so that we can re-write \eqref{eq:howard-stitched} as
\begin{equation}
  \P \left [ \exists k \in \NN : \left \lvert \sum_{i=1}^k Y_i \right \rvert \geq b_k \right ] \leq \exp \left (-m^{2(1-\lambda)/q - 1} \right ) \sum_{j=0}^\infty \frac{2}{h(j)} \leq  H \exp \left (-m^{2(1-\lambda)/q - 1} \right ), 
\end{equation}
where $H := 2\sum_{j=0}^\infty \exp ( -(e^j)^{2(1-\lambda)/q - 1} ) \geq \sum_{j=0}^\infty 2/h(j)$ and the boundary $b_k$ is given by 
\begin{equation}
  b_k := \eps m^{\lambda/q} \sqrt{c^2 k \left ( (k + m)^{2(1-\lambda) /q-1} +  m^{2(1-\lambda)/q - 1} \right )},
\end{equation}
Notice  we have  for any $k \geq m$,
\begin{align}
         b_k \leq \eps m^{\lambda/q}\sqrt{2 c^2k   (k+m)^{2(1-\lambda)/q-1}} \leq \eps \sqrt{2} c m^{\lambda/q} (k+m)^{(1-\lambda)/q}
\leq \eps \sqrt{2}c (k+m)^{1/q}.
\end{align}
Therefore, we have the following time-uniform bound on $\supkm |k^{-1/q} \sum_{i=1}^k Y_i|$:
\begin{align}
  \P \left [ \sup_{k \geq m} \frac{1}{k^{1/q}} \left \lvert \sum_{i=1}^k Y_i \right \rvert \geq 2^{1/q+1/2} c \eps \right ] &= \P \left [ \sup_{k \geq m} \frac{1}{(2k)^{1/q}} \left \lvert \sum_{i=1}^k Y_i \right \rvert \geq \sqrt{2} c \eps \right ]\\
  \verbose{
                 &\leq \P \left [ \sup_{k \geq m} \frac{1}{(k + m)^{1/q}} \left \lvert \sum_{i=1}^k Y_i \right \rvert \geq \sqrt{2} c\eps \right ] \\
  }
                     &\leq \P \left [ \exists k \geq 1 :  \left \lvert \sum_{i=1}^k Y_i \right \rvert \geq  \sqrt{2}c \eps (k+m)^{1/q} \right ] \\
  &\leq H \exp \left (-m^{2(1-\lambda)/q - 1} \right ).
\end{align}
Notice that $H$ can be upper bounded as follows:
\begin{align}
  H = 2\sum_{j=0}^\infty \exp \left (- e^{ j(2(1-\lambda)/ q - 1) } \right )
    \leq 2 + 2 \int_{0}^\infty \exp \left ( - e^{ y(2(1-\lambda)/ q - 1) } \right )\dd y
  = 2 + \frac{2\int_{1}^\infty t^{-1} e^{-t} \dd t}{2(1-\lambda)/q - 1}.
\end{align}
Note that $\int_1^\infty t^{-1} e^{-t} \dd t \leq e^{-1} < 1/2$.
Observing that $2^{1/q + 1/2} c \leq 2(e^{1/4} + e^{-1/4}) < 4.13$, we can write the time-uniform bound on $\supkm |k^{-1/q} \sum_{i=1}^k Y_i|$ as
\begin{align}
  \P \left [ \sup_{k \geq m} \frac{1}{k^{1/q}} \left \lvert \sum_{i=1}^k Y_i \right \rvert \geq 4.13 \eps \right ] &\leq \left ( 2 + \frac{q}{2(1-\lambda) - q} \right ) \exp \left(-m^{2(1-\lambda)/q - 1} \right ),
\end{align}
which completes Step I.

\subsubsection*{Step II: Time-uniform concentration of $\sum_{i=1}^k Z_i$.}

For this step, we will use the following lemma.
\begin{lemma}\label{lemma:maximal weighted sum inequality}
  Let $\infseqn{a_n}$ be a monotonically nondecreasing and strictly positive sequence and let $\infseqn{b_n}$ be any real sequence. Then for any  $M \in \NN$, we have
  \begin{equation}
    \max_{ 1\leq k \leq M} \frac{|\sum_{i=1}^k b_i |}{a_k} \leq 2 \max_{1 \leq k \leq M} \left \lvert \sum_{i=1}^k \frac{b_i}{a_i} \right \rvert. \end{equation}
\end{lemma}
\begin{proof}
  For $n \in \NN$, define $b_n' := b_n / a_n$, 
  \begin{equation}
     T_n := \sum_{i=1}^n b_i, \quad\text{and}\quad T_{n}' := \sum_{i=1}^n b_i'.
  \end{equation}
  Using summation by parts, we can write $T_k$ for each $k \in \NN$ as
  \begin{align}
    T_k = \sum_{i=1}^k a_ib_i' = \sum_{i=1}^k a_i(T_i' - T_{i-1}') = a_k T_k' - \sum_{i=1}^k (a_i - a_{i-1}) T_{i-1}',
  \end{align}
  where we set $a_0 := 0$.
  Therefore, for each $k \in \NN$, 
  \begin{align}
    |T_k| &\leq a_k |T_k'| +  \sum_{i=1}^k (a_i - a_{i-1}) |T_{i-1}'| \leq a_k |T_k'| +  \left (\max_{1 \leq j \leq k} |T_{j-1}'|  \right )\sum_{i=1}^k (a_i - a_{i-1}) 
    \leq 2 a_k \max_{1 \leq j \leq k} |T_j'|.
  \end{align}
  Dividing both sides by $a_k$ and taking maxima over $k \in \{1, \dots, M\}$ completes the proof.
\end{proof}

We now continue with Step II. By monotone convergence, \cref{lemma:maximal weighted sum inequality}, and Kolmogorov's inequality, we have 
\begin{align}
  \P \left [ \supkm \frac{1}{k^{1/q}} \left \lvert \sum_{i=1}^k Z_i \right \rvert \geq \eps \right ] &\leq \lim_{M\nearrow \infty} \P \left [ \max_{1 \leq k \leq M} \left \lvert \sum_{i=1}^k \frac{Z_i}{i^{1/q}} \right \rvert \geq \frac{\eps}{2} \right ]
  \leq \frac{4}{\eps^2} \EE_\P \left [ \sum_{i=1}^\infty \frac{Z_i^2}{i^{2/q}} \right ].
\end{align}
Since $\EE_\P \left [ Z_i^2 \right ] \leq \EE_\P \left [ X_1^2 \1 \{ \eps m^{\lambda/q} < |X_1| \leq (i-1)^{1/q} \} \right ]$ for each $i \in \NN$, we get 
\begin{align}
  \EE_\P \left [ \sum_{i=1}^\infty \frac{Z_i^2}{i^{2/q}} \right ] &= \sum_{i=1}^\infty \EE_\P \left [ \frac{Z_i^2}{i^{2/q}} \right ] \\
  &\leq \EE_\P \left [ X_1^2 \1 \{ |X_1| > \eps m^{\lambda/q} \} \sum_{i=1}^\infty \1 \{ |X_1|^q \leq i-1 \} \frac{1}{i^{2/q}} \right ]\\
  &\leq \EE_\P \left [ X_1^2 \1 \{ |X_1| > \eps m^{\lambda/q} \} \int_{|X_1|^q}^\infty  \frac{1}{y^{2/q}}\dd y \right ]\\
  \verbose{
  &= \EE_\P \left [ X_1^2 \1 \{ |X_1| > \eps m^{\lambda/q} \} \left ( -\frac{2}{q} + 1 \right ) \left ( \underbrace{\lim_{y \nearrow \infty} y^{-2/q + 1}}_{=0} - [|X_1|^q]^{-2/q + 1} \right ) \right ]\\
  }
  &= \EE_\P \left [ \cancel{X_1^2} \1 \{ |X_1| > \eps m^{\lambda/q} \} \frac{q}{2-q} \cancel{X_1^{-2}} |X_1|^q \right ]\\
 &\leq 
 \frac{q}{2-q}\U_\P^\brackq(\eps^q m^\lambda)
\end{align}
Putting the above arguments together, we have that
\begin{align}
  \P \left [ \supkm \frac{1}{k^{1/q}} \left \lvert \sum_{i=1}^k Z_i \right \rvert \geq \eps \right ] &\leq \frac{4q}{\eps^2(2-q)}\U_\P^\brackq(\eps^q m^\lambda).
\end{align}

\subsubsection*{Step III: Time-uniform concentration of $\sum_{i=1}^k R_i$.}
For any $i \in \NN$, set $\mu_i := \EE_\P \left [ X_i \1 \{ |X_i| > \eps m^{\lambda/q} \lor (i-1)^{1/q} \} \right ]$ so that $R_i$ can be written more succinctly as
\begin{equation}
  R_i = X_i \1 \{ |X_i|^q > \eps^q m^\lambda \lor (i-1) \} - \mu_i.
\end{equation}
Hence we can upper-bound the desired tail probability $\P \left [ \supkm k^{-1/q} | \sum_{i=1}^k R_i | \geq \eps \right ]$ as 
\begin{align}
  \P \left [ \supkm \frac{1}{k^{1/q}} \left \lvert \sum_{i=1}^k R_i \right \rvert \geq \eps \right ] &\leq \underbrace{\1 \left \{ \supkm \frac{1}{k^{1/q}} \left \lvert \sum_{i=1}^k \mu_i \right \rvert \geq \eps \right \}}_{(\star)}  + \underbrace{\P \left [ \exists i \in \mathbb{N} : |X_i|^q > \eps^q m^\lambda \lor (i-1) \right ]}_{(\dagger)}.
\end{align}
Upper-bounding the indicator $(\star)$, we have
\begin{align}
(\star) &\leq \underbrace{\1 \left \{ \supkm \frac{1}{k^{1/q}} \EE_\P \left [ X_1 \1 \{ |X_1|^q > \eps^q m^\lambda \} \right ] \geq \frac{\eps}{2} \right \}}_{(\star i)}\\
  &\quad+ \underbrace{\1 \left \{ \supkm  \frac{1}{k^{1/q}} \sum_{i=2}^k \EE_\P \left [ \frac{(i-1)^{1/q}X_i}{(i-1)^{1/q}} \1 \{ |X_i|^q > \eps^q m^\lambda \lor (i-1) \} \right ] \geq \frac{\eps}{2} \right \}}_{(\star ii)}.
\end{align}
Notice now that since $m \in \NN$, we have that $|X_1| / \eps \leq |X_1|^q / \eps^q$ on the event $\{|X_1|^q > \eps^q m^\lambda\}$ and thus $(\star i)$ can be upper bounded as
\begin{align}
  (\star i) \leq 2\EE_\P\left [\frac{|X_1|}{\eps} \1 \{ |X_1|^q > \eps^q m^\lambda \} \right ] \leq \frac{2}{\eps^q} \U_\P^\brackq(\eps^q m^\lambda).
\end{align}
Turning to the second term $(\star ii)$, notice that on the event $\{ |X_i|^q > i-1 \} $, we have  $|X_i| / (i-1)^{1/q} \leq |X_i|^q / (i-1)$ by virtue of the fact that $q \in [1,2)$, and hence $(\star ii)$ can be upper-bounded as
\begin{align}
  (\star ii) &\leq \1 \left \{ \supkm \frac{1}{k^{1/q}} \sum_{i=2}^k \EE_\P \left [ \frac{(i-1)^{1/q}|X_1|^q}{i-1} \1 \{ |X_1|^q > \eps^q m^\lambda \lor (i-1) \} \right ] \geq \frac{\eps}{2} \right \} \\
  &\leq \1 \left \{ \supkm \frac{1}{k^{1/q}} \sum_{i=2}^k \EE_\P \left [ (i-1)^{1/q-1}|X_1|^q \1 \{ |X_1|^q > \eps^q m^\lambda \} \right ] \geq \frac{\eps}{2}\right \}\\
&\leq 
               \1 \left \{ \supkm \frac{\EE_\P \left [ |X_1|^q \1 \{ |X_1|^q > \eps^q m^\lambda \} \right ]}{k^{1/q}} \int_{1}^{k+1} (y-1)^{1/q-1}\dd y \geq \frac{\eps}{2}\right \}\\
             &= \1 \left \{ \supkm \frac{\EE_\P \left [ |X_1|^q \1 \{ |X_1|^q > \eps^q m^\lambda \} \right ]}{\cancel{k^{1/q}}} q \cancel{k^{1/q}} \geq \frac{\eps}{2} \right \}\\
    &\leq \frac{4}{\eps}\U_\P^\brackq(\eps^q m^\lambda).
\end{align}
Putting the bounds on $(\star i)$ and $(\star ii)$ together, we have that
\begin{equation}
  (\star) \leq (\star i ) + (\star ii) \leq \left ( \frac{2}{\eps^q} + \frac{4}{\eps} \right ) \U_\P^\brackq(\eps^q m^\lambda).
\end{equation}
Turning now to $(\dagger)$, we note that if $|X_i|^q > \eps^q m^\lambda \lor (i-1)$, then $|X_i|^q \1 \{ |X_i|^q > \eps^q m^\lambda \} > i-1$ and thus we can union bound to obtain
\begin{align}
  (\dagger) &\leq \sum_{i=1}^\infty \P \left [  |X_1|^q > \eps^q m^\lambda \lor (i-1)  \right ]
  \leq \P \left [ |X_1|^q > \eps^q m^\lambda \right ] + \sum_{i=2}^\infty \P \left [ |X_1|^q \1 \{ |X_1|^q > \eps^q m^\lambda \} > i-1 \right ].
\end{align}
Now, clearly 
\[
    \P [|X_1|^q > \eps^q m^\lambda] \leq \EE_\P[|X_1|^q / (\eps^q m^\lambda) \1 \{ |X_1|^q > \eps^q m^\lambda \}] \leq \eps^{-q}\U_\P^\brackq(\eps^q m^\lambda),
\]
so we can continue the above upper bound as
\begin{align}
(\dagger) &\leq \frac{1}{\eps^{q}}\U_\P^\brackq(\eps^q m^\lambda)+ \sum_{i=1}^\infty \P \left [ |X_1|^q \1 \{ |X_1|^q > \eps^q m^\lambda \} > i \right ]\\
              &\leq \frac{1}{\eps^q} \U_\P^\brackq(\eps^q m^\lambda) + \int_0^\infty \P [|X_1|^q \1 \{ |X_1|^q > \eps^q m^\lambda \} > u]\dd u\\
  &= \left ( 1 + \frac{1}{\eps^q} \right )\U_\P^\brackq(\eps^q m^\lambda).
\end{align}
Putting the bounds on $(\star)$ and $(\dagger)$ together, we obtain that
\begin{align}
  \P \left [ \supkm \frac{1}{k^{1/q}} \left \lvert \sum_{i=1}^k R_i \right \rvert \geq \eps \right ] &\leq \left ( 1 + \frac{4}{\eps} + \frac{3}{\eps^q} \right ) \U_\P^\brackq(\eps^q m^\lambda) \leq \frac{8}{\eps^2 \land 1} \U_\P^\brackq(\eps^q m^\lambda).
\end{align}

\subsubsection*{Step IV: Combining the concentration bounds of Steps I--III.}
Finally, combining the results from the previous three steps yields
\begin{align}
  &\P \left [ \supkm \frac{1}{k^{1/q}} \left \lvert \sum_{i=1}^k X_i  \right \rvert \geq 6.13\eps \right ] \leq c_{q,\lambda} \exp \left ( -m^{2(1-\lambda)/q - 1} \right )+ \frac{12q}{(\eps^2 \land 1)(2-q)} \U_\P^\brackq(\eps^q m^\lambda),
\end{align}
where $c_{q, \lambda} :=  (4(1-\lambda) - q)/(2(1-\lambda) - q)$. Taking $\lambda = 1/2 - q/4$ and replacing $\eps$ by $\eps / 6.13$ completes the proof.
\end{proof}

\section{Proofs of \cref*{theorem:lil} and \cref*{corollary:self-normalized-lil}}\label{proof:lil}
  
\begin{proof}[Proof of \cref{theorem:lil}]
  Fix $\eps > 0$. Let $V_n := \sum_{i=1}^n \left ( X_i^2 + 2\sigmasqub \right )/3$ and $\psi$ be the function given by $\psi(\lambda) := \lambda^2 / 2$, $\lambda \in \RR$. By \citet[Lemma 3(f)]{howard_exponential_2018}, we have that $\infseqn{X_n}$ is a sub-Gaussian process with cumulative variance proxy $\infseqn{V_n}$, meaning that the exponential process
\begin{equation}
  \exp \left ( \lambda S_n - \psi(\lambda) V_n \right ) %
\end{equation}
is upper-bounded by a nonnegative supermartingale starting at one with respect to the filtration $(\Fcal_n)_{n\in \NN_0}$ generated by the process $(X_n)_{n \in \NN}$. Therefore, by \citet[Theorem~1]{howard2018uniform}, for $\eta = 1+\eps$ and any $v_0 > \sigmasqub$ (to be set later), 
\begin{equation}\label{eq:invoking-howard-theorem1}
  \P \left [ \exists n \in \NN : V_n  \geq v_0 ~\text{and}~ S_n \geq b(V_n)  \right ] \leq \sum_{j=\lfloor \log_\eta (v_0 / \sigmasqub) \rfloor}^\infty \frac{1}{h(j)},
\end{equation}
where $h(j) := (j+1)^\eta \zeta(\eta)$ so that $\sum_{j=0}^\infty 1/h(j) = 1$ %
by construction, and where
\begin{equation}
  b(v) := \frac{\eta^{1/4} + \eta^{-1/4}}{\sqrt{2}}  \sqrt{v  \left ( \eta \log \left ( \log_\eta \left ( \eta v / \sigmasqub \right ) \right ) + \log \left ( 2\zeta(\eta) \right )\right )}, \qquad v > 0.
\end{equation}
Note that whenever $v \geq \sigmasqub$, the boundary $b(v)$ can be upper-bounded as
\begin{align}
  b(v) \leq \xi \sqrt{v \eta \left ( \log \log \left ( \eta v/\sigmasqub \right ) + \ell_\eps \right )},
\end{align}
where we write $\xi := (\eta^{1/4} + \eta^{-1/4}) / \sqrt{2}$ and $\ell_\eps := \log(2 \zeta(\eta) / ( \log(\eta)))$ for succinctness. Putting the above upper bound together with \eqref{eq:invoking-howard-theorem1}, we have for any $v_0 > \sigmasqub$,
\begin{align}
  \P \left [ \exists n \in \NN : V_n  \geq v_0 ~\text{and}~ S_n \geq \xi  \sqrt{V_n  \eta \left (  \log \left ( \log \left ( \eta V_n / \sigmasqub \right ) \right ) + \ell_\eps \right )}  \right ] \leq \sum_{j=\lfloor \log_\eta (v_0 / \sigmasqub) \rfloor}^\infty \frac{1}{h(j)}. \label{eq:invoking-howard-refined}
\end{align}
Consider next the event $A_m$ for each $m \in \NN$ given by
\begin{equation}
  A_m := \left \{ \forall k \geq m,~ \left \lvert \frac{1}{k} \sum_{i=1}^k X_i^2 - \sigma_\P^2 \right \rvert < \eps \sigmasqub \right \}.
\end{equation}
Notice that by definition of $V_k$, we have, on the event $A_m$, for every $k \geq m$,
\begin{equation}
  V_k = \frac{1}{3}\sum_{i=1}^k (X_i^2 + 2 \sigmasqub) < \frac{1}{3} \left ( k (\sigma_\P^2 + \eps \sigmasqub) + 2k(\sigmasqub + \eps \sigmasqub) \right ) \leq k (1+\eps) \sigmasqub.
\end{equation}
Moreover, we have, for every $k \in \NN$,
\begin{equation}
    V_k = \frac{1}{3} \sum_{i=1}^k (X_i^2 + 2 \sigmasqub) \geq \frac{2k}{3} \sigmasqub,
\end{equation}
and hence the following sequence of inequalities:
\begin{align}
  &\P \left [ \supkm  \frac{|S_k / \widebar \sigma|}{ \xi\sqrt{(1+\eps)^2   k \left ( \log \log ( (1 + \eps )^2 k) + \ell_\eps \right )}} 
  \geq 1 \right ]\\
  =\ &\P \left [ \exists k \geq 1 : k \sigmasqub  \geq m \sigmasqub  \text{ and }  \frac{|S_k/\widebar \sigma|}{\xi \sqrt{(1+\eps) k  \eta \left (\log \log ( \eta (1 + \eps) k  ) + \ell_\eps \right )}} 
  \geq 1 \right ]\\
  \leq\  &\P \left [ \exists k \geq 1 : V_k \geq \frac{2}{3} m\sigmasqub \text{ and } \frac{ |S_k|}{ \xi \sqrt{V_k \eta \left ( \log \log ( \eta V_k / \sigmasqub) + \ell_\eps \right )}} 
           \geq 1 \right ]\\
  &\quad + \P \left [ \supkm  \left \lvert \frac{1}{k} \sum_{i=1}^k X_i^2  - \sigma_\P^2 \right \rvert \geq \eps \sigmasqub \right ].
\end{align}
Analyzing the first term, we apply \eqref{eq:invoking-howard-refined} with $v_0 = 2m\sigmasqub / 3$, noticing that since $m \in \NN \setminus \{1,2\}$, we have  $v_0 \geq 4 \sigmasqub /3 > \sigmasqub$ and hence 
\begin{align}
  \P \left [ \exists k \geq 1 : V_k \geq \frac{2}{3} m \sigmasqub \text{ and }\frac{ |S_k|}{\xi\sqrt{V_k \eta  \left ( \log \log (\eta V_k / \sigmasqub ) + \ell_\eps \right )}} 
    \geq 1 \right ] &\leq \sum_{j=\lfloor \log_\eta (2m/3) \rfloor}^\infty \frac{2}{h(j)} \land 1.
\end{align}
When $\log_\eta(2m/3) < 1$, the upper bound is given by 1, so assume $\log_\eta(2m/3) \geq 1$ in what remains. We have
\begin{align}
   \P \left [ \exists k \geq 1 : V_k \geq \frac{2}{3} m \sigmasqub \text{ and }\frac{ |S_k|}{\xi\sqrt{V_k \eta  \left ( \log \log (\eta V_k / \sigmasqub ) + \ell_\eps \right )}} \geq 1 \right ]  &\leq \int_{\lfloor\log_\eta(2m/3) \rfloor - 1}^\infty \frac{2}{(x+1)^\eta \zeta(\eta)} \dd x \land 1\\
             &\leq \frac{ 2^{\eta}\log_\eta^{1-\eta}(2m/3)}{(\eta-1)\zeta(\eta)}.
\end{align}
Analyzing the second term, we have by \cref{theorem:l1-concentration},
\begin{align}
\P \left [ \supkm \left \lvert  \frac{1}{k} \sum_{i=1}^k X_i^2 - \sigma_\P^2 \right \rvert \geq \eps \sigmasqub \right ] \leq \frac{262}{(\widebar \sigma^4 \eps^2) \wedge 1} \left ( m^{2\lambda - 1} + \EE_\P \left [|X_1^2 - \sigma_\P^2| \1 \{ |X_1^2 - \sigma_\P^2| \geq m^\lambda \}\right ] \right ).
\end{align}
Combining these two estimates and recalling that $\eta = 1+ \eps$, we get 
\begin{align}
  &\P \left [ \supkm  \frac{| S_k/ \widebar \sigma|}{ \xi\sqrt{(1 + \eps)^2 k \left ( \log \log ( (1 + \eps )^2 k) + \ell_\eps \right )}} 
  \geq 1 \right ] \\
  \leq\ &\frac{2^{1+\eps}\log_{1+\eps}^{-\eps}(2m/3) }{\eps \zeta(1+\eps)} + \frac{262}{(\widebar \sigma^4 \eps^2) \wedge 1} \left ( m^{2\lambda - 1} + \EE_\P[|X_1^2 - \sigma_\P^2| \1 \{ |X_1^2 - \sigma_\P^2| \geq m^\lambda \}] \right ).
\end{align}
Letting $c_\eps := (1+\eps)\xi$ yields the desired result and completes the proof of \cref{theorem:lil}.\footnote{Note that we are employing \cref{theorem:l1-concentration} in an intermediate step of the proof of \cref{theorem:lil}. One may wonder why we do not use the improved \cref{theorem:lq-concentration} instead since certain polynomially vanishing terms are replaced by some that vanish exponentially fast. The reason for this is because \cref{theorem:lil} (and other iterated logarithm inequalities) have an additional term that vanishes logarithmically in $m$ which will always dominate any polynomial (or exponential) term.}
\end{proof}

\begin{proof}[Proof of \cref{corollary:self-normalized-lil}]
Fix $\eps > 0$ and $m \in \NN \setminus \{1,2\}$. Applying \cref{theorem:lil} to the random variables $\infseqn{X_n/\sigma_\P}$ with $\sigmasqub = 1$, we have 
\begin{align}
  &\P \left [ \supkm \frac{ |S_k/ \sigma_\P|}{c_\eps \sqrt{k  \left ( \log \log ( (1 + \eps )^2 k) + \ell_\eps \right )}} 
  \geq 1 \right ] \\
  \leq\ &\frac{2^{1+\eps}}{\eps \log_{1+\eps}^{\eps}(2m/3) \zeta(1+\eps)} + \frac{262}{\eps^2 \wedge 1} \left ( m^{2\lambda - 1} + \EE_\P[|(X_1/ \sigma_\P)^2 - 1| \1 \{ |(X_1/ \sigma_\P)^2 - 1| \geq m^\lambda \}] \right )\\
  \leq\ &\frac{2^{1+\eps}}{\eps \log_{1+\eps}^{\eps}(2m/3) \zeta(1+\eps)} + \frac{262}{\eps^2 \wedge 1} \left (m^{2\lambda - 1} + \U_\P^\bracktwo(m^\lambda)\right ),
\end{align}
where the last inequality uses that $m^\lambda > 1$.
Notice that for each $n \geq m$ we can write the difference $\widehat \sigma_n^2 - \sigma_\P^2$ as
\begin{align}
  \widehat \sigma_n^2 - \sigma_\P^2 = \frac{1}{n} \sum_{i=1}^n (X_i^2 - \widehat \mu_n^2) - \EE_\P[X_1^2] = \frac{1}{n} \sum_{i=1}^n X_i^2 - \EE_\P[X_1^2] - \widehat \mu_n^2.
\end{align}
Letting $\eps' = \eps \land 1/4$, have that on the event  
\begin{equation}
  A := \bigcap_{k=m}^\infty \left \{ \left \lvert \frac{1}{k} \sum_{i=1}^k \frac{X_i^2}{\sigma_\P^2} - 1\right \rvert < \eps' \text{  and  } \left \lvert \frac{1}{k} \sum_{i=1}^k \frac{X_i}{\sigma_\P} \right \rvert < \sqrt{\eps'} \right \},
\end{equation}
it holds that $|\widehat \sigma_k^2/ \sigma_\P^2 - 1| < 2\eps'$ and hence $|\sigma_\P / \widehat \sigma_k| < (1 - 2\eps')^{-1/2}$ for all $k \geq m$.
Applying \cref{theorem:l1-concentration} to $\infseqn{X_n/\sigma_\P}$ and $\infseqn{X_n^2 / \sigma_\P^2 - 1}$, respectively, we have 
\begin{align}
  \P \left [ \supkm  \left \lvert \frac{1}{k}\sum_{i=1}^k \frac{X_i}{\sigma_\P} \right \rvert \geq \sqrt{\eps'} \right ] &\leq \frac{262}{\eps \land (1/4)} \left(m^{2\lambda - 1} + \EE_\P\left[\left|\frac{X_1}{ \sigma_\P}\right| \1 \{ |X_1 / \sigma_\P| \geq m^\lambda \}\right]\right) \\
  &\leq \frac{262}{\eps^2 \land (1/4)} (m^{2\lambda - 1} + \U_\P^\bracktwo(m^\lambda))
\end{align}
and
\begin{align}
  \P \left [ \supkm \left \lvert \frac{1}{k} \sum_{i=1}^k \frac{X_i^2}{\sigma_\P^2} - 1 \right \rvert \geq \eps' \right ] &\leq \frac{262}{\eps^2 \land (1/16)} \left(m^{2\lambda - 1} + \EE_\P\left[\left|\frac{X_1^2}{ \sigma_\P^2}-1\right|  \1 \{ |X_1^2 / \sigma_\P^2 - 1|\geq m^\lambda \}\right]\right)\\
  &\leq \frac{262}{\eps^2 \land (1/16)} (m^{2\lambda - 1} + \U_\P^\bracktwo(m^\lambda)).
\end{align}
Putting all of the previous inequalities together and recalling that on the event $A$, it holds $|\sigma_\P / \widehat \sigma_k| < (1-2\eps')^{-1/2}$, we have
\begin{align}
  &\P \left [ \supkm  \frac{| S_k / \widehat \sigma_k|(1 - 2\eps')^{1/2}}{c_\eps \sqrt{ k  \left ( \log \log ( (1 + \eps )^2 k) + \ell_\eps \right )}}  \geq 1 \right ] \\
  =\ &\P \left [ \supkm \frac{|\sigma_\P / \widehat \sigma_k|}{(1-2\eps')^{-1/2}} \frac{  |S_k / \sigma_\P|}{c_\eps  \sqrt{k  \left ( \log \log ( (1 + \eps )^2 k) + \ell_\eps \right )}} \geq 1 \right ]\\
  \leq\ &\P \left [ \supkm  \frac{ |S_k / \sigma_\P|}{c_\eps \sqrt{ k  \left ( \log \log ( (1 + \eps )^2 k) + \ell_\eps \right )}} \geq 1 \right ] + \P [A^c]\\
  \leq\ &\frac{2^{1+\eps}\log^{-\eps}_{1+\eps}(2m/3)}{\eps  \zeta(1+\eps)} + \frac{786}{\eps^2 \land (1/16)} \left ( m^{2\lambda - 1} + \U_\P^\bracktwo (m^\lambda) \right ),
\end{align}
completing the proof of \cref{corollary:self-normalized-lil}.
\end{proof}

%% file: composite.tex
\section{Composite strong laws and the inverse capital measure}\label{section:implications-composite}

A so-called ``composite'' generalization of Ville's theorem for families of probability measures was recently introduced to the literature by \citet{ruf2022composite}. Their generalization of a zero-probability event to a class of measures $\Pcal$ is given in terms of the \emph{inverse capital outer measure} $\nu_\Pcal$ defined for any $A \in \Fcal_\infty$, where $(\Fcal_n)_{n \in \NN_0}$ is some filtration, by
\begin{equation}\label{eq:inverse capital outer measure}
  \nu_\Pcal[A] := \inf_{\tau \in \Tcal : A \subseteq \{\tau < \infty \}} \sup_\Pin \P [\tau < \infty],
\end{equation}
where $\Tcal$ is the set of all $(\Fcal_n)_{n \in \NN_0}$-stopping times. As suggested by the name, $\nu_\Pcal$ is not a measure but  an outer measure.
In short, \citet[Theorem~3.1]{ruf2022composite} states that for an event $A \in \Fcal$, $\nu_\Pcal[A] = 0$ if and only if there exists a $\Pcal$-$e$-process diverging to $\infty$ pathwise on $A$. Such a result differs from \cref{theorem:distribution-uniform-ville} since the latter requires that the $e$-process additionally diverges pathwise to $\infty$ tail-uniformly. We illustrate the gap between these two notions of divergence in the context of the SLLN.

In \citet[Theorem~4.3]{ruf2022composite}, the authors prove a composite SLLN which states that if $\lim_\xto \supP\U_\P(x) = 0$, then there exists a $\Pcal$-$e$-process that diverges pathwise to $\infty$ on $A_\onediv$, equivalently, $\nu_\Pcal[A_\onediv] = 0$.
The authors conjecture that \emph{``some condition like [uniform integrability] is necessary to restrict $\Pcal$''} in order to conclude that $\nu_\Pcal[A_\onediv]= 0$. We now demonstrate that uniform integrability is not a necessary condition by constructing a family $\Pcal^\star$ for which $X_1$ is not $\Pcal^\star$-uniformly integrable alongside a $\Pcal^\star$-$e$-process that diverges on $A_\onediv$. 
Indeed, for each $b \in \NN$, let
$\P_b \in \Pcalbar$ be the probability measure so that $X_1$ takes the values $\pm b$ with equal $\P_b$-probability. In other words, $\P_b[X_1 = x] = (1/2)^{\1 \{ x \in \{-b,b\} \}}$ for any $x$. Letting $\Pcal^\star = \{ \P_b : b \in \NN \}$, we see that $X_1$ is not $\Pcal^\star$-uniformly integrable since for any $x \geq 0$, $\EE_{\P_b} [|X_1| \1 \{ |X_1| \geq x \}] = b$ if $b \geq x$ and $0$ otherwise. Therefore,
\begin{equation}
  \sup_{b \in \NN}\EE_{\P_b} [|X_1| \1 \{ |X_1| \geq x \}] = \infty
\end{equation}
for any $x \geq 0$. Nevertheless, consider the process 
\begin{equation}
  E_n^\star := \sum_{j\in \NN} \1 \left \{ \text{there exists $k\in
\{m_j, \ldots, n\}$ such that } \frac{|S_k|}{k(|X_1| + \1 \{X_1=0\})} \geq \frac{1}{j}  \right \}
\end{equation}
where 
\[
m_j := \min \{ m \in \NN \setminus \{1\} : 262 j^2  (m^{-1/3} + \U_{\P_1}(m^{1/3}) ) \leq 2^{-j} \}; \qquad j \in \NN,
\]
and where $\U_{\P_1}$ is defined in \cref{theorem:l1-concentration}.
Since $|S_k|/|X_1|$ under $\P_b$ has the same distribution as $|S_k|$ under $\P_1$, and by the same arguments as in \cref{SS:5.1}, but with \cref{theorem:lq-concentration} replaced by \cref{theorem:l1-concentration},
$(E_n^\star)_{n \in \NN_0}$ is a $\Pcal^\star$-$e$-process. Alternatively, one can view $(E_n^\star)_{n \in \NN_0}$ as ``waiting'' to see $X_1$, at which point $\P_b$ conditionally on $|X_1|$ equals $\P_{|X_1|}$ and is known exactly. Hence the composite behavior of $\supkm |S_k| / k$ under $\Pcal^\star$ can be reduced to pointwise behavior of $|S_k| / (k|X_1|)$ under $\P_1$.
\begin{remark}
 The event $A_\onediv$ can be represented through two different event lattices:
\begin{equation}
  A_\onediv = \bigcup_{\eps > 0} \bigcap_{m=1} \left \{ \supkm \frac{|S_k|}{k}  \geq \eps \right \}
  = \bigcup_{\eps > 0} \bigcap_{m=1} \left \{ \supkm \frac{|S_k|}{k(|X_1| + \1\{ X_1=0 \})}  \geq \eps \right \}.
\end{equation}
Note that the former lattice is the canonical one implicitly considered by \citet{chung_strong_1951} and which can be found in \cref{corollary:chung-slln,corollary:game-theoretic sllns}. By \cref{corollary:game-theoretic sllns} combined with the fact that $X_1$ is not $\Pcal^\star$-uniformly integrable, we have that $\infseqn{E_n^\star}$ cannot diverge tail-uniformly on the former lattice. However, it is easy to check that it does on the latter.
\end{remark}

Nevertheless, once combined with the fact that tail-uniform divergence to $\infty$ implies pointwise divergence to $\infty$, \cref{corollary:game-theoretic sllns} yields an extension of \citet[Theorem~4.3]{ruf2022composite} to $q^\tth$ moments for $q \in (1,2)$ and at the Marcinkiewicz--Zygmund rate of $o(n^{1/q-1})$. Similarly, \cref{corollary:game-theoretic lils} yields a composite LIL for random variables with finite second moments. We state these results here for the sake of completeness.
\begin{corollary}[Composite Marcinkiewicz--Zygmund-type pathwise SLLNs and a composite pathwise LIL]\label{corollary:inverse capital measure implications}
  Let $\Pcal \subset \Pcalbar$ and $q \in [1,2)$. If $\supP \U_\P^\brackq(x) \to 0$ as $\xto$, then $\nu_\Pcal[A_\qdiv] = 0$. Furthermore, if $\supP \U_\P^\bracktwo(x) \to 0$ as $\xto$, then $\nu_\Pcal[A_\fluc] = 0$.
\end{corollary}